\documentclass[11pt,reqno]{amsart}

\usepackage{amscd,amsmath,amssymb,amsthm,amsfonts,verbatim,graphicx,enumitem,cancel,latexsym,mathrsfs}
\usepackage[colorlinks]{hyperref}

\usepackage{tikz-cd,tikz,pgfplots}
\usetikzlibrary{positioning}
\pgfplotsset{compat=1.18}
\usetikzlibrary{decorations.pathreplacing}

\newcommand{\C}{\mathbb{C}}
\newcommand{\R}{\mathbb{R}}

\renewcommand{\epsilon}{\varepsilon}
\renewcommand{\leq}{\leqslant}
\renewcommand{\geq}{\geqslant}

\theoremstyle{plain}
\newtheorem{theorem}{Theorem}[section]
\newtheorem{lemma}[theorem]{Lemma}
\newtheorem{corollary}[theorem]{Corollary}
\newtheorem{proposition}[theorem]{Proposition}

\theoremstyle{definition}
\newtheorem{definition}[theorem]{Definition}
\newtheorem{example}[theorem]{Example}
\newtheorem{remark}[theorem]{Remark}

\numberwithin{equation}{section}
\begin{document}
	
\title{Convex functions with symplectic Hessian}
	
\begin{abstract}
We prove a third-order derivative estimate for convex solutions to the real Monge-Ampère equation ${\rm det}\,{\rm Hess}(u) = 1$ on an open set in $\mathbb{R}^{2m}$ under the additional assumption that ${\rm Hess}(u)$ lies in ${\rm Sp}(2m,\mathbb{R})$ at every point. Our method is a geometric interpretation and extension to higher dimensions of Nitsche's classical proof of the Bernstein theorem for the real Monge-Ampère equation on $\mathbb{R}^2$. For $m = 1$ we also improve Nitsche's constant as well as some estimates due to Calabi, and we construct examples of solutions with interesting geometric behavior.
\end{abstract}
	
\author{José Rafael Santiago Arellano}
\address{Mathematisches Institut, Universit\"at M\"unster, 48149 M\"unster, Germany}
\email{\href{mailto:joser.s@uni-muenster.de}{joser.s@uni-muenster.de}, \href{mailto:joserafaelsar@gmail.com}{joserafaelsar@gmail.com}}

\date{\today}
	
\maketitle
	
\markboth{Convex functions with symplectic Hessian}{José Rafael Santiago Arellano}

\section{Introduction}

Our main result in this paper is

\begin{theorem}\label{thm:main}
Let $u: U \to \mathbb{R}$ be a smooth function on an open set $U \subset \mathbb{R}^{2m}$ such that ${\rm Hess}(u)$ is positive definite and in ${\rm Sp}(2m,\mathbb{R})$ everywhere. Then the scalar curvature $R$ of the associated Hessian metric $g$ on $U$, i.e. $g|_x = {\rm Hess}(u)|_x$ for all $x \in U$, satisfies $R \geq 0$ and
\begin{equation}\label{eq:main_est}\sqrt{R(x)} \leq \frac{3m}{{\rm dist}_{g|_x}(x,\partial U)}\end{equation}
for all $x \in U$. In particular, if $U = \mathbb{R}^{2m}$, then $R \equiv 0$ and consequently $u$ is a quadratic polynomial. The constant $3m$ may not be optimal but cannot be made smaller than $\sqrt{m}$.
\end{theorem}

Calabi \cite{EC} famously proved such a result for $U \subset \mathbb{R}^n$, $n =2,3,4,5$, and ${\rm Hess}(u)$ positive definite and in ${\rm SL}(n,\mathbb{R})$ everywhere, with the constant $3m$ replaced by $M_n <  127, 2660, 40100, 653000$, respectively. This can be viewed as a $C^3$ a priori estimate for solutions to the real Monge-Ampère equation because (due to some cancellations) $\sqrt{R(x)}$ is then actually a norm of the third derivative of $u$ at $x$. Calabi also pointed out that for $n = 2$, Nitsche's version \cite{Nit} of Jörgens's proof \cite{Jor} of the Bernstein theorem for the Monge-Ampère equation on $\mathbb{R}^{2}$ via complex analysis yields $M_2 \leq 4$. We revisit Nitsche's proof, improve the constant from $4$ to $3$, and generalize it under the symplectic Hessian condition, which is equivalent to being unimodular for $m = 1$ but is of course much stronger for $m > 1$.

Our method is to consider the Kähler metric with  potential $u \circ {\rm Re}$ on the domain $U + i\mathbb{R}^{2m} \subset \mathbb{C}^{2m}$. This is Ricci-flat if ${\rm Hess}(u) \in {\rm SL}(2m,\mathbb{R})$ and it is hyper-Kähler if ${\rm Hess}(u) \in {\rm Sp}(2m,\mathbb{R})$. Then we use the Pedersen-Poon ansatz \cite{PP} to express $u$ in terms of a holomorphic map from a domain in $\mathbb{C}^m$ to the Siegel upper half-plane in $\mathbb{C}^{m \times m}$. The existence of such a representation follows indirectly from \cite{ACD,BC,VC,Freed,Hit1,Hit2}. However, we need very explicit formulas, especially with regard to the domains of the functions considered, which we have been unable to find in the literature. In Section \ref{sec:general} we develop such formulas and then deduce Theorem \ref{thm:main} using Cauchy's inequality from complex analysis.

In Section \ref{sec:special} we construct examples and discuss some more special consequences, mainly for $m = 1$. It turns out that our geometric approach also provides some new insights into Calabi's proof in this case, e.g. improved numerical constants and a hint as to how this proof may perhaps be extended to dimensions $n \geq 6$. (Recall that Calabi's theorem was in fact extended to all dimensions by Pogorelov \cite{Pogo}, but Pogorelov's method is different and it is an open problem to make Calabi's original method work in higher dimensions.) Our main results in Section \ref{sec:special} are as follows:

\begin{theorem}\label{thm:main:3_not_sharp}
    For $m=1$, there exists no example realizing the bound $M_2 \leq 3$ of Theorem \ref{thm:main}.
\end{theorem}

\begin{theorem}\label{main:curv_diffeq}
Let $u: U \to \mathbb{R}$ be a smooth function on an open set in $\mathbb{R}^{2}$ such that ${\rm Hess}(u)$ is positive definite and ${\rm det}\,{\rm Hess}(u) = 1$. Let $g$ denote the associated Hessian metric on $U$ with Laplace-Beltrami operator $\Delta$ and scalar curvature $R \geq 0$.   Then, for all $\alpha >0$,
$$\Delta R^{\alpha}\geq 3\alpha R^{\alpha+1}$$
holds in the barrier sense on $U$ and equality can be attained at a point where $R > 0$.
\end{theorem}

This inequality was proved by Calabi \cite{EC} for $\alpha = \frac{1}{2}$. In \cite[p. 114]{EC} Calabi also claimed that equality cannot be attained at a point where $R > 0$ but we were unable to reproduce his argument. We show that Calabi's proof of the inequality generalizes to all exponents $\alpha \geq \frac{1}{2}$. For $\alpha < \frac{1}{2}$ a new argument based on the holomorphic representation of $u$ is needed. As a consequence of being able to let $\alpha \to 0$ in the inequality, we can improve a curvature estimate from \cite{EC} in dimension $2$:

\begin{theorem}\label{thm:improved:calabi}
There exists a constant $c \leq 1.16$ with the following property. In Theorem \ref{main:curv_diffeq}, let $C(s)$, $s \in [0,\gamma)$, be a minimal unit-speed geodesic with respect to $g$ in $U$. Then
\begin{equation}\label{eq:glob:curv:geo}\sqrt{{\rm Ric}(C'(s),C'(s))} \leq \frac{c}{\gamma-s}\end{equation}
holds for all $s \in [0,\gamma)$. Moreover, $c$ cannot be made smaller than $1.11$.
\end{theorem}

In \cite{EC}, Calabi proved the bound $c \leq 1.52$ and asked \cite[p. 122]{EC} whether $c < \sqrt{2}$ is possible because this improvement would eliminate a difficult lemma due to Levinson \cite[p. 123]{EC} from the proof of the main result of \cite{EC} in dimension $2$. A variant of Levinson's lemma that we prove in this paper is:

\begin{proposition}\label{prop:levin:mod}
For any Riemannian manifold $M^n$ and minimal unit-speed geodesic $C: [0,\gamma) \to M$,
\begin{equation}\label{eq:levin}\liminf_{s \to \gamma} \;(\gamma-s)^2{\rm Ric}(C'(s),C'(s)) \leq \frac{n-1}{4}.\end{equation}
Equality can hold on a surface of positive curvature.
\end{proposition}

Now consider the Hessian manifold $(U,g)$ associated with a convex solution of $\det\hspace{0.25mm}{\rm Hess}(u) = 1$ on an open set $U \subset \mathbb{R}^n$. One of the key points of Calabi's work \cite{EC} is that if \eqref{eq:levin} can be improved to
\begin{equation}\label{eq:calabi:dream}\limsup_{s \to \gamma}\;(\gamma - s)^2{\rm Ric}(C'(s),C'(s)) < \frac{n}{n-1},\end{equation}
then his third-order estimate for $u$ follows with a purely dimensional constant $M_n < \infty$. In particular, Bernstein's theorem for the Monge-Ampère equation holds in dimension $n$. If $\frac{n-1}{4} < \frac{n}{n-1}$, i.e. $n \leq 5$, the original version of Levinson's lemma, which then implies \eqref{eq:calabi:dream} in an integrated sense, suffices for this conclusion. For $n = 2$, our global curvature estimate \eqref{eq:glob:curv:geo} along $C(s)$ trivially implies \eqref{eq:calabi:dream}.

It is a very interesting question whether \eqref{eq:calabi:dream} or a slight weakening in an integrated sense holds for Monge-Ampère manifolds of dimension $6$ or higher, the point being that the Monge-Ampère structure plays no role in the proof of Levinson's lemma. We end this paper with Examples \ref{ex:ric:ma}, \ref{ex:schwarz:2} and \ref{ex:2nd:order:harmonic}, which heuristically suggest that, for Monge-Ampère manifolds of dimension $n = 2$, the left-hand side of \eqref{eq:calabi:dream} might always be bounded by $\frac{2}{9} < \frac{n-1}{4}$. (Incidentally these examples also show that $M_2 \geq 1$, that $c \geq 1.11$ and that Proposition \ref{prop:levin:mod} may fail if $C$ is not minimal.) It would be very interesting to make this heuristic rigorous or to find an analogous heuristic in higher dimensions.

\subsection*{Acknowledgments}

This paper is based on the author's PhD thesis \cite{JSAR} at the University of Münster, supervised by Hans-Joachim Hein and Bianca Santoro. This PhD project was funded by the Deutsche Forschungsgemeinschaft (DFG, German Research Foundation) under Germany's Excellence Strategy EXC 2044/2–390685587, Mathematics Münster: Dynamics–Geometry–Structure" and by the Alexander von Humboldt Foundation through the Humboldt Professorship awarded to Gustav Holzegel.

\section{General formulas and proof of Theorem \ref{thm:main}}\label{sec:general}

\subsection{Semi-flat Calabi-Yau metrics}

The starting point of our computation are the following simple observations. These are all either well-known or folklore and are included here just for clarity. Details of the computations can be found in \cite[Chapter 4.1]{JSAR}.

\begin{lemma}\label{MAMG}
    Let $x^i + \sqrt{-1}x^i_*$, $i = 1, \ldots, n$, be the standard linear coordinates on $\mathbb{C}^n$. Let $U\subset \R^n$ be an open set. Let $u:U\to \R$ be a smooth convex solution of the real Monge-Ampère equation 
    \begin{equation}\label{MA}
    \det \hspace{0.25mm} {\rm Hess}(u)=1.
    \end{equation}
    Then the manifold $M:=U + \sqrt{-1}\R^n\subset\C^{n}$ is a Ricci-flat Kähler manifold with the data
  $$J_M\left( \frac{\partial}{\partial x^i}\right)=\frac{\partial}{\partial x_{*}^{i}}, \quad \omega_M=\displaystyle \sum_{i,j=1}^n u_{x^{i}x^{j}}dx^{i}\wedge dx^{j}_{*},\quad g_M=\displaystyle \sum_{i,j=1}^n u_{x^{i}x^{j}}(dx^{i}\otimes dx^{j}+dx_{*}^{i}\otimes dx_{*}^{j}),$$
with $n$ commuting, pointwise linearly independent, holomorphic Killing vector fields $\frac{\partial}{\partial x^i_*}$. 
\end{lemma}

We now follow the conventions of Besse \cite[Chapter 9]{Besse}.  The projection $\pi: M\to U$, $\pi(x,x_{*})=x$, is a Riemannian submersion with respect to $g_M$ and the Monge-Ampère metric $g$ on $U$, $g|_x = {\rm Hess}(u)|_x$ for all $x \in U$, with trivial horizontal distribution $\mathscr{H} = TU \oplus \{0\}$.

  \begin{lemma}\label{TA}
The O'Neill tensors of the Riemannian submersion $\pi$ are given by
$$T=\frac{1}{2}g^{k\ell}u_{x^{i}x^{j}x^{\ell}}dx_{*}^{i}\otimes dx^j\otimes \partial_{x^{k}_{*}}-\frac{1}{2} g^{k\ell}u_{x^{i}x{^j}x^{\ell}}dx_{*}^{i}\otimes dx_{*}^{j}\otimes \partial_{x^{k}},\quad A\equiv 0.$$
Moreover, the mean curvature $N$ of the fibers of $\pi$, i.e. the vertical trace of $T$, vanishes.
  \end{lemma}

The property $N = 0$ follows from the given formula for $T$ by using the linearization of the Monge-Ampère equation \eqref{MA}, which says that $u_{x^\ell}$ is a harmonic function with respect to the metric $g$ on $U$ for every $\ell$. However, it is also a well-known folklore observation that the fibers of $\pi$ are special Lagrangian submanifolds of the Calabi-Yau manifold $M$ with its parallel holomorphic volume form $$\Omega_M = (dx^1 + \sqrt{-1} dx^1_*) \wedge \cdots \wedge (dx^n + \sqrt{-1} dx^n_*),$$ so they are volume-minimizing under compactly supported perturbations, hence minimal.

Since ${\rm Ric}_{g_M} \equiv 0$, the O'Neill formulas now immediately yield the following:

\begin{lemma}\label{SCT}
  The Ricci tensor and scalar curvature of the Monge-Ampère metric $g$ on $U$ are given by
 $${\rm Ric}_g(X,Y)=g(T X,T Y), \quad R_g =\frac{1}{4}{}g^{ij}g^{k\ell}g^{mn}u_{x^{i}x^{k}x^{m}}u_{x^{j}x^{\ell}x^{n}}.$$
In particular, the metric $g$ has non-negative Ricci curvature.
\end{lemma}

Calabi \cite[Eq. 2.8]{EC} proved the above formula for $R_g$ by direct computations on $U$. 

\subsection{Partial Legendre transform of Kähler metrics with symmetries}

\begin{definition}
    Let $\langle\cdot,\cdot\rangle$ denote the standard inner product on $\mathbb{R}^n$. Let $f:U\subset \R^{k} \times \R^{n-k}\to \R$ be a $C^{1}$ function such that the map $x \mapsto \frac{\partial f}{\partial x}(x,y)$ is injective for every fixed $y$. Then the \emph{partial Legendre transform} of $f$ is the function $g:V\subset \R^{k} \times \R^{n-k}\to \R$ determined by
    \begin{itemize}
    \item $V=T_{1}(U)$, where 
    $$T_{1}(x,y):=\left(\frac{\partial f}{\partial x}(x,y),y\right),\;\text{and}$$
    \item $f(x,y)+g(\tilde{x},\tilde{y})=\langle x,\tilde{x} \rangle$ for all $(x,y) \in U$, where $(\tilde{x},\tilde{y}) = T_1(x,y)$.
    \end{itemize}
\end{definition}

We now use this concept to give a convenient description of Kähler metrics with sufficiently many symmetries. This is a small extension of the Pedersen-Poon ansatz \cite{PP}, without assuming from the outset that the metric is hyper-Kähler. Since it is very important for us to have precise formulas, we will develop this ansatz carefully, deferring some tedious computations to \cite[Chapter 4.2]{JSAR}.

Let $(M,\omega_M,J_M,g_M)$ be a Kähler manifold of real dimension $4m$. We suppose that $M$ admits $m$ commuting and pointwise linearly independent holomorphic Killing vector fields $X_{i}$. 
We furthermore assume that there exist
local holomorphic coordinates $(w_{i},z_{i})$ on $M$ such that
$$X_{i}=\frac{\partial}{\partial \mathrm{Im}(w_{i})},$$
and such that there exists a local Kähler potential $K$ for $\omega_M$ which is invariant under the action of the holomorphic Killing vector fields, i.e.
$$K_{w_{i}}=K_{\overline{w}_{i}}.$$
In these coordinates, the Kähler form $\omega_M$ takes the form
$$\omega_M=\frac{\sqrt{-1}}{2} \sum_{i,j=1}^{m}K_{w_{i}\overline{w}_j}dw_{i}\wedge d\overline{w}_j+K_{w_{i}\overline{z}_j}dw_{i}\wedge d\overline{z}_j+K_{z_{i}\overline{w}_j}dz_{i}\wedge d\overline{w}_j+K_{z_{i}\overline{z}_j}dz_{i}\wedge d\overline{z}_j.$$
Importantly, the matrix $(K_{w_i \overline{w}_j})$ is now real, so in particular symmetric and positive definite. Viewing $K$ as a function of the coordinates ${\rm Re}(w_i)$ and $z_i$, we let $\Lambda/4$ denote the partial Legendre transform of $K/4$ with respect to the coordinates ${\rm Re}(w_i)$. Thus,
\begin{equation}\label{LT}
    \frac{\Lambda}{4}=-\frac{K}{4}+ \sum_{i=1}^{m} \tilde{x}_{i}{\rm Re}(w_i),
\end{equation}
where the new coordinates $(\tilde{x}_i,\tilde{z}_{i})$ are given by
$$\tilde{x}_{i}:=\frac{K_{{\rm Re}(w_i)}}{4}=\frac{K_{w_{i}}}{2},\quad  \tilde{z}_{i}:={z}_{i}.$$
Note that, after shrinking its domain, the map $({\rm Re}(w_i),z_i) \mapsto (\tilde{x}_i, \tilde{z}_i)$ will be a $C^1$ diffeomorphism onto its image. Indeed, its Jacobian has the form $(\begin{smallmatrix}A & \ast \\ 0 & I_{2m}\end{smallmatrix})$, where
\begin{equation}\label{Jacobian}
A_{ij} := \frac{\partial \tilde{x}_{i}}{\partial {\rm Re}(w_{j})} = K_{w_i w_j} = K_{w_i \overline{w}_j}.
\end{equation}
As noted above, this is a real, symmetric, positive definite matrix, so it is in particular invertible.

By differentiating \eqref{LT} with respect to ${\rm Re}(w_{j})$ and using the chain rule, we obtain
\begin{equation}\label{fw}\sum_{i=1}^{m}\left(\frac{\partial \Lambda}{\partial \tilde{x}_{i}}-4\mathrm{Re}(w_{i})\right)\frac{\partial \tilde{x}_{i}}{\partial {\rm Re}(w_{j})}=0  \quad \mbox{for} \quad j=1,\ldots,m.\end{equation}
Observe that the coefficient matrix of this linear system is given by $A^\top$ with $A$ as in \eqref{Jacobian}, so it is invertible.
 Thus, \eqref{fw} is equivalent to
\begin{equation}\label{Rew}\frac{\partial \Lambda}{\partial \tilde{x}_{i}}=4\mathrm{Re}(w_{i}) \quad \text{for}  \quad i = 1, \ldots, m.\end{equation}
From this we may deduce the following, where upper indices denote matrix inversion:
\begin{align}\begin{split}\label{KF}
   K_{w_{i}\overline{w}_{j}}  & =  4\Lambda^{\tilde{x}_{i}\tilde{x}_{j}}, \\
   K_{z_{i}\overline{w}_{j}}  & =  -2\Lambda_{\tilde{z}_i \tilde{x}_k}\Lambda^{\tilde{x}_{k}\tilde{x}_{j}},\\
   K_{z_{i}\overline{z}_{j}}  & =  -\Lambda_{\tilde{z}_{i}\overline{\tilde{z}}_{j}}+\Lambda_{\tilde{z}_{i}\tilde{x}_{k}}\Lambda^{\tilde{x}_{k}\tilde{x}_{\ell}}\Lambda_{\tilde{x}_{\ell}\overline{\tilde{z}}_{j}}.
\end{split}\end{align}
For instance, by differentiating \eqref{Rew} by $\tilde{x}_j$ and using \eqref{Jacobian}, we get $\Lambda_{\tilde{x}_i \tilde{x}_j} = 4A^{ij}$, which proves the first line of \eqref{KF}. The other two lines are more complicated. By applying the operators $\partial,\overline{\partial}$ defined by the holomorphic coordinates $(w_i,z_i)$ to \eqref{LT}, roughly as in \cite[Eq. 3.142]{HKLR}, we obtain first
\begin{align*}\begin{split}
   \partial K  & =  \displaystyle \partial\left(-\Lambda+\sum_{i=1}^{m} 4\tilde{x}_{i}\mbox{Re}(w_{i}) \right) \\
   & =  \displaystyle \sum_{i=1}^{m}-\frac{\partial \Lambda}{\partial w_{i}}dw_{i}-\frac{\partial \Lambda}{\partial z_{i}}dz_{i}+4\mbox{Re}(w_{i})\partial\tilde{x}_{i}+4\tilde{x}_{i}\partial\mbox{Re}(w_{i})\\
&= \displaystyle \sum_{i=1}^{m} -\frac{\partial \Lambda}{\partial w_{i}}dw_{i}-\frac{\partial \Lambda}{\partial z_{i}}dz_{i}+\frac{\partial \Lambda}{\partial \tilde{x}_i}\left( \sum_{j=1}^{m} \frac{\partial \tilde{x}_{i}}{\partial w_{j}}dw_{j}+\frac{\partial \tilde{x}_i}{\partial z_j}dz_j\right)+2\tilde{x}_{i}dw_{i}.
\end{split}\end{align*}
Now observe that 
$$\sum_{i=1}^{m}\frac{\partial \Lambda}{\partial \tilde{x}_i}\frac{\partial \tilde{x}_{i}}{\partial w_{j}}= \frac{\partial \Lambda}{\partial w_j},\quad   \quad \sum_{i=1}^m \frac{\partial \Lambda}{\partial\tilde{x}_i}\frac{\partial \tilde{x}_i}{\partial z_j} = \frac{\partial \Lambda}{\partial z_j} - \frac{\partial \Lambda}{\partial \tilde{z}_j}.$$
Thus,
\begin{align*}\begin{split}\partial K&=\sum_{i=1}^{m}-\frac{\partial \Lambda}{\partial \tilde{z}_{i}}dz_{i}+2\tilde{x}_{i}dw_{i},\\
     \overline{\partial}\partial K & = \sum_{i=1}^m -\overline{\partial}\left( \frac{\partial \Lambda}{\partial \tilde{z}_{i}}\right)\wedge d z_{i}  + 2\overline{\partial}\tilde{x}_{i}\wedge dw_{i}\\
     & = \sum_{i,j=1}^m -\Lambda_{\tilde{z}_i\overline{z}_j}d\overline{z}_j \wedge dz_i - \Lambda_{\tilde{z}_i \overline{w}_j}d\overline{w}_j \wedge dz_i + 2\frac{\partial\tilde{x}_i}{\partial \overline{z}_j}d\overline{z}_j \wedge dw_i + 2\frac{\partial \tilde{x}_i}{\partial \overline{w}_j} d\overline{w}_j \wedge dw_i.
\end{split}\end{align*}
Comparing coefficients, we deduce that 
\begin{align}
K_{z_i\overline{w}_j} &=
-\Lambda_{\tilde{z}_i\overline{w}_j},\label{WTF1}\\
K_{z_i\overline{z}_j} &= -\Lambda_{\tilde{z}_i\overline{z}_j}.\label{WTF2}
\end{align}
From \eqref{WTF1} and the chain rule, using also $\Lambda_{\tilde{x}_i\tilde{x}_j} = 4A^{ij}$, we get the second line of \eqref{KF}. 
From \eqref{WTF2} and the chain rule, using also the second line of \eqref{KF}, we get the third line of \eqref{KF}.

We now rewrite \eqref{KF} in a more compact way. By calling
$$\Phi=(\Phi_{ij}):=\frac{1}{4}(\Lambda_{\tilde{x}_{i}\tilde{x}_{j}}) ,\quad  \Gamma=(\Gamma_{ij}):=\frac{1}{2}(\Lambda_{\tilde{x}_{i}\overline{\tilde{z}}_{j}}),$$
we find that
\begin{align}
{\rm Hess}_{\C}(K)  = \begin{pmatrix}
     (K_{w_{i}\overline{w}_{j}})  & (K_{w_{i}\overline{z}_{j}})  \\
     (K_{z_{i}\overline{w}_{j}})  & (K_{z_{i}\overline{z}_{j}})
    \end{pmatrix}
     =  \begin{pmatrix}
   \Phi^{-1}  & -\Phi^{-1}\Gamma   \\
   - \overline{\Gamma}^{\top}\Phi^{-1}  & -(\Lambda_{\tilde{z}_{i}\overline{\tilde{z}}_{j}}) +\overline{\Gamma}^{\top} \Phi^{-1} \Gamma
\end{pmatrix}.\label{legtrf}
   \end{align}
The complex Hessian of $K$ defines a Hermitian metric
$$h_M=\sum_{i,j=1}^{m} K_{w_{i}\overline{w}_j}dw_{i}\otimes d\overline{w}_j+K_{w_{i}\overline{z}_j}dw_{i}\otimes d\overline{z}_j+K_{z_{i}\overline{w}_j}dz_{i}\otimes d\overline{w}_j+K_{z_{i}\overline{z}_j}dz_{i}\otimes d\overline{z}_j,$$
which encodes the Riemannian metric $g_M= {\rm Re}(h_M)$ associated with $\omega_M$. 
Equation \eqref{legtrf} expresses the coefficients in terms of ${\rm Hess}(\Lambda)$. To complete the partial Legendre transform, we also need to convert $dw_i, dz_i$ into the new coordinates
$$\tilde{x}_{i}:= \frac{K_{{\rm Re}(w_i)}}{4} =\frac{K_{w_{i}}}{2},\quad   \tilde{s}_{i}:=\mathrm{Im}(w_{i})  ,\quad  \tilde{y}_{i}:=\mathrm{Re}(z_{i}) ,\quad  \tilde{t}_{i}:=\mathrm{Im}(z_{i}).$$
By \eqref{Rew}, the Jacobian of this coordinate change is given by
\begin{align*}
dz_i & =  d\tilde{y}_i + \sqrt{-1}d\tilde{t}_i,\\
dw_i & =  d{\rm Re}(w_i) + \sqrt{-1}d\tilde{s}_i\nonumber\\
& =  \left(\sum_{k=1}^m \frac{1}{4}\Lambda_{\tilde{x}_i\tilde{x}_k}d\tilde{x}_k + \frac{1}{4}\Lambda_{\tilde{x}_{i}\tilde{y}_{k}}d\tilde{y}_{k}+\frac{1}{4}\Lambda_{\tilde{x}_{i}\tilde{t}_{k}}d\tilde{t}_{k}  \right)+ \sqrt{-1}d\tilde{s}_i\nonumber\\
 & =  \left(\sum_{k=1}^m \Phi_{ik}d\tilde{x}_{k} + {\rm Re}(\Gamma_{ik})d\tilde{y}_{k}+{\rm Im}(\Gamma_{ik})d\tilde{t}_{k} \right)+ \sqrt{-1}d\tilde{s}_i .
\end{align*}
After a long calculation similar to \cite[Eq. 3.13--3.19]{PP}, we deduce the following formula for $g_M$:

\begin{theorem}\label{CH}
Let $(M,\omega_M,J_M,g_M)$ be a Kähler manifold of real dimension $4m$ with $m$ Killing vector fields. Assume $(w_{i},z_{i})$ are local holomorphic coordinates such that these Killing vector fields take the form $\partial_{{\rm Im}(w_i)}$. Assume $K$ is a local Kähler potential of $\omega_M$ independent of the coordinates $\mathrm{Im}(w_{i})$. Then, after shrinking the coordinate domain if necessary, ${\rm Hess}_{\C}(K)$ can be expressed in terms of the partial Legendre transform $\Lambda/4$ of $K/4$ with respect to the coordinates $\mathrm{Re}(w_{i})$\textup{:}
\begin{align}{\rm Hess}_{\C}(K)=\begin{pmatrix}
   \Phi^{-1}  & -\Phi^{-1}\Gamma   \\
   - \overline{\Gamma}^{\top}\Phi^{-1}  & -(\Lambda_{\tilde{z}_{i}\overline{\tilde{z}}_{j}}) +\overline{\Gamma}^{\top} \Phi^{-1} \Gamma
\end{pmatrix},\label{eq:cx:hessian:leg}\end{align}
where
$$\Phi=(\Phi_{ij}):=\frac{1}{4}(\Lambda_{\tilde{x}_{i}\tilde{x}_{j}}) ,\quad  \Gamma=(\Gamma_{ij}):=\frac{1}{2}(\Lambda_{\tilde{x}_{i}\overline{\tilde{z}}_{j}}),$$
with $\Lambda$ defined in the coordinates $(\tilde{x}_{i},\tilde{z}_{i})$,
$$\tilde{x}_{i}:= \frac{K_{{\rm Re}(w_i)}}{4} = \frac{K_{w_{i}}}{2},\quad  \tilde{z}_{i}:=z_{i},\quad  
    \frac{\Lambda}{4}:=-\frac{K}{4}+ \sum_{i=1}^{m} \tilde{x}_{i}{\rm Re}(w_i).
$$
Moreover, the Riemannian metric $g_M$ associated with $\omega_M$ takes the form
\begin{align*}
g_M & = \sum_{i,j=1}^m \biggl(\Phi_{ij} d\tilde{x}_{i}\otimes d\tilde{x}_{j}-\displaystyle \frac{\Lambda_{\tilde{y}_{i}\tilde{y}_{j}}+\Lambda_{\tilde{t}_{i}\tilde{t}_{j}}}{4}(d\tilde{y}_{i}\otimes d\tilde{y}_{j} + d\tilde{t}_{i} \otimes d\tilde{t}_{j})  \\
     &  \; \quad  \quad  \quad  \quad  \quad  \quad \quad \quad \, \, \, \, \, \,  +\displaystyle
    \frac{\Lambda_{\tilde{y}_{i}\tilde{t}_{j}}-\Lambda_{\tilde{t}_{i}\tilde{y}_{j}}}{4}(d\tilde{t}_{i}\otimes d\tilde{y}_{j}-d\tilde{y}_{i}\otimes d\tilde{t}_{j})+\Phi^{ij}(d\tilde{s}_{i}+\theta_{i})\otimes(d\tilde{s}_{j}+\theta_{j})\biggr), \nonumber
\end{align*}
with
$$\tilde{x}_{i}:= \frac{K_{{\rm Re}(w_i)}}{4} =\frac{K_{w_{i}}}{2},\quad   \tilde{s}_{i}:=\mathrm{Im}(w_{i})  ,\quad  \tilde{y}_{i}:=\mathrm{Re}(z_{i}) ,\quad  \tilde{t}_{i}:=\mathrm{Im}(z_{i}),$$
and
$$\theta_{i}:=\sum_{k=1}^{m}\frac{ \sqrt{-1}}{2}( \overline{\Gamma}_{ik}d\tilde{z}_{k}-\Gamma_{ik}d\overline{\tilde{z}}_{k}).$$
\end{theorem} 

\begin{corollary}\label{MAP}
Under the assumptions and conclusions of Theorem \ref{CH},
$${\rm Hess}_{\C}(K)\in {\rm Sp}(2m,\C) \Longleftrightarrow \begin{cases}
\textup{(i)} &\Gamma=\Gamma^{\top}\ \text{and}\\
\textup{(ii)} &\Phi=-(\Lambda_{\tilde{z}_{i}\overline{\tilde{z}}_{j}}).
\end{cases}
$$
In this situation,
$${\rm Hess}_{\C}(K)=\begin{pmatrix}
   \Phi^{-1}  & -\Phi^{-1}\Gamma   \\
   - \overline{\Gamma}\Phi^{-1}  & \Phi +\overline{\Gamma} \Phi^{-1} \Gamma
\end{pmatrix}$$
and
$$g_M = \sum_{i,j=1}^m \Phi_{ij} (d\tilde{x}_{i}\otimes d\tilde{x}_{j}+d\tilde{y}_{i}\otimes d\tilde{y}_{j} + d\tilde{t}_{i} \otimes d\tilde{t}_{j}) + \Phi^{ij}(d\tilde{s}_{i}+\theta_{i})\otimes(d\tilde{s}_{j}+\theta_{j}).$$
\end{corollary}

\begin{proof}
The matrix \eqref{eq:cx:hessian:leg} lies in ${\rm Sp}(2m,\mathbb{C})$ if and only if
\begin{align}
\label{SP1}
   \Phi^{-1}(-\overline{\Gamma}^\top \Phi^{-1}) - (-\Phi^{-1}\overline{\Gamma})\Phi^{-1}& = 0,\\
    \label{SP2}
    \Phi^{-1}(-(\Lambda_{\tilde{z}_{i}{\overline{\tilde{z}}}_{j}}) +\overline{\Gamma}^{\top} \Phi^{-1} \Gamma) - (-\Phi^{-1}\overline{\Gamma})(-\Phi^{-1}\Gamma) & =  I_m,\\
\label{SP3}
    (-\Gamma^\top \Phi^{-1})(-(\Lambda_{\tilde{z}_{i}{\overline{\tilde{z}}}_{j}}) +\overline{\Gamma}^{\top} \Phi^{-1} \Gamma) - (-(\Lambda_{\overline{\tilde{z}}_{i}{{\tilde{z}}}_{j}}) +\Gamma^{\top} \Phi^{-1} \overline{\Gamma})(-\Phi^{-1}\Gamma) &=  0.
\end{align}
If these three equations hold, then the first one implies $\Gamma = \Gamma^\top$; and, using this, we easily deduce from the second that $\Phi=-(\Lambda_{\tilde{z}_{i}\overline{\tilde{z}}_{j}})$. Thus, (i) and (ii) hold. Conversely, (i) implies the first equation above, and (i) and (ii) together imply the second and the third.

Under these conditions, we may clearly simplify the formula for ${\rm Hess}_{\C}(K)$ as claimed. Now observe that $\Phi$ is real by construction. Thus, we conclude from (ii) that $(\Lambda_{\tilde{z}_{i}\overline{\tilde{z}}_{j}})$ is real. This implies that
$$\Lambda_{\tilde{z}_{i}\overline{\tilde{z}}_{j}}=\frac{\Lambda_{\tilde{y}_{i}\tilde{y}_{j}}+\Lambda_{\tilde{t}_{i}\tilde{t}_{j}}}{4} ,\quad  \Lambda_{\tilde{y}_{i}\tilde{t}_{j}}-\Lambda_{\tilde{t}_{i}\tilde{y}_{j}}=0.$$
Therefore, and because $\Lambda_{\tilde{z}_{i}\overline{\tilde{z}}_{j}} = -\Phi_{ij}$, the metric $g_M$ takes the stated form.
\end{proof}

\begin{remark}\label{polyharm}
  Condition (ii) of Corollary \ref{MAP} can be rewritten as 
  $$\Lambda_{\tilde{x}_i\tilde{x}_j} + 4 \Lambda_{\tilde{z}_i\overline{\tilde{z}}_j} = 0 \quad \text{for all} \quad i,j=1,\ldots,m.$$
  For a general function $\Lambda: U \to \R$ ($U \subset \R^m \times \C^m$ open) and for $(\tilde{x}_i,\tilde{z}_i)$ the standard linear coordinates on $\R^m \times \C^m$, this condition is often called \emph{polyharmonicity}, see \cite[p. 505]{Bie} and \cite[Eq. 2.1]{HKLR}. If $\Lambda$ is polyharmonic, then so are all the functions $\Phi_{ij} := \frac{1}{4} \Lambda_{\tilde{x}_i\tilde{x}_j} = \Phi_{ji}$. In the setting of Corollary \ref{MAP}, and assuming (ii), $\Phi$ is thus a symmetric, positive definite, polyharmonic matrix. This property generalizes the positivity and harmonicity of the conformal factor in the Gibbons-Hawking ansatz for $m = 1$.
\end{remark}

Corollary \ref{MAP} is a refinement of the Pedersen-Poon ansatz \cite{PP}. We are effectively characterizing the case that $(g_M,J_M)$ admits a compatible hyper-Kähler structure \emph{such that the holomorphic symplectic form $\sum dw_i \wedge dz_i$ is parallel}. This leads to the extra bit of information that $\Gamma = \Gamma^\top$.

\subsection{Semi-flat hyper-Kähler metrics: set-up}

Let $U\subset\R^{2m}$ be a domain. Let $u:U\to \R$ be a convex smooth solution of the Monge-Ampère equation \eqref{MA} such that ${\rm Hess}(u)\in\text{Sp}(2m,\R)$. Write the standard coordinates of $\mathbb{R}^{2m}$ as $x_1,y_1,\ldots,x_m,y_m$. By Lemma \ref{MAMG}, the manifold $M=U\times \R^{2m}$ is a Ricci-flat Kähler manifold with $2m$ commuting Killing vector fields $\partial_{s_i}, \partial_{t_i}$, where we write the standard coordinates on the $\mathbb{R}^{2m}$ factor as $s_1,t_1,\ldots,s_m,t_m$. Specifically, $M$ has holomorphic coordinates
$$w_{i}:=x_{i}+\sqrt{-1} s_{i},\quad    z_{i}:=y_{i}+\sqrt{-1} t_{i},$$ 
and a Kähler form $\omega_M$ defined by the Kähler potential $K(w_{i},z_{i}):=4u(x_{i},y_{i})$, i.e.
$$\omega_M := \frac{\sqrt{-1}}{2}\partial\overline\partial K,$$
giving rise to the Ricci-flat Kähler metric 
$$g_{M}:=\omega(-,J_M-).$$ 
Since $K$ is independent of the coordinates $s_{i}$ and $t_{i}$, the complex Hessian of $K$ with respect to the holomorphic coordinates $(w_{i},z_{i})$ coincides with the Hessian of $u$, i.e.
\begin{align*}
    {\rm Hess}_{\C}(K)={\rm Hess}(u)=\left(\begin{array}{cc}
       (u_{x_{i}x_{j}}) & (u_{x_{i}y_{j}}) \\
        (u_{y_{i}x_{j}}) & (u_{y_{i}y_{j}})
    \end{array}\right) \in \R^{2m\times 2m}.
\end{align*}

Fix an arbitrary point of $U$. After translating the coordinates we may assume that this is the origin $O \in \mathbb{R}^{2m}$. After subtracting an affine function from $u$, we may assume that $u(O)=0$ and $\nabla u(O) = 0$. Set $A:={\rm Hess}(u)|_{O}\in {\rm Sp}(2m,\mathbb{R})$. Since $A$ is symmetric and positive definite, we may use $A^{-1/2}$ as a linear coordinate transformation to arrange that ${\rm Hess}(u)|_O = {I}_{2m}$.

\begin{lemma}
    If $A \in {\rm Sp}(2m,\mathbb{R})$ is symmetric and positive definite, and if a function $f: \mathbb{R}^+ \to \mathbb{R}^+$ satisfies $f(\lambda^{-1}) = f(\lambda)^{-1}$ for all $\lambda \in \mathbb{R}^+$, then $f(A) \in {\rm Sp}(2m,\mathbb{R})$.
\end{lemma}

\begin{proof}
    Since $A$ is symplectic, we have that $A^\top J A = J$ for the standard complex structure $J$ on $\mathbb{R}^{2m}$. Thus, if $v$ is an eigenvector of $A$ with eigenvalue $\lambda$, then $Jv$ is an eigenvector of $A$ with eigenvalue $\lambda^{-1}$. Thus, we can write $A = Q \Lambda Q^{-1}$, where $Q \in SO(2m)$ commutes with $J$ and $$\Lambda = {\rm diag}(\lambda_1,\lambda_1^{-1},\ldots, \lambda_m, \lambda_m^{-1}), \quad \lambda_1,\ldots,\lambda_m > 0.$$
In particular, $\Lambda J = J \Lambda^{-1}$, and hence 
$f(\Lambda)J = J f(\Lambda^{-1}) = J f(\Lambda)^{-1}$.
Using this and the property that $QJ = JQ$, one easily checks that $f(A) = Q f(\Lambda) Q^{-1} \in {\rm Sp}(2m,\mathbb{R})$, as desired.
\end{proof}

Thus, the linear map $A^{-1/2}$ used to normalize ${\rm Hess}(u)|_O$ is itself symplectic, so this normalization preserves the condition ${\rm Hess}(u) \in {\rm Sp}(2m,\mathbb{R})$ everywhere. Thus, without loss of generality,
\begin{equation}\label{symplectic}
u(O)=0, \quad (\nabla u)(O)=0, \quad  \operatorname{Hess}(u)|_{O}=I_{2m}, \quad \operatorname{Hess}(u)\in{\rm Sp}(2m,\R).
\end{equation}

For the rest of this section, all results are considered over the ball $B_{\delta(O)}(O)\subset U$, where, following Calabi's notation from \cite{EC}, $\delta$ denotes the affine distance to $\partial U$, i.e. $\delta(x) := {\rm dist}_{g|_x}(x,\partial U)$ for all $x \in U$. Because of our normalization \eqref{symplectic}, $\delta(O)$ is simply the standard Euclidean distance from $O$ to $\partial U$.

Fix the set of $m$ holomorphic Killing vector fields $\partial_{s_i}$, $i = 1, \ldots, m$. By Theorem \ref{CH}, and using the fact that $K$ is independent of the coordinates $s_{i},t_{i}$, the partial Legendre transform $\Lambda/4$ of $K/4$ with respect to the coordinates $x_i = {\rm Re}(w_i)$ is defined in the new coordinates $\tilde{x}_{i},\tilde{y}_{i}$,
$$\tilde{x}_{i}:=\frac{K_{x_i}}{4} = \frac{K_{w_{i}}}{2}=u_{x_{i}},\quad  \tilde{y}_{i}:=y_{i},\quad  \frac{\Lambda}{4} := -\frac{K}{4} + \sum_{i=1}^m \tilde{x}_i x_i.$$
Notice that the map
\begin{equation}\label{mapT1}
T_{1}:B_{\delta(O)}(O)\to T_{1}(B_{\delta(O)}(O)),  \quad (\tilde{x}_{i},\tilde{y}_{i}) = T_{1}(x_{i},y_{i}) := (u_{x_{i}},y_{i}),
\end{equation}
is in fact a diffeomorphism: its Jacobian is invertible,
\begin{align}J_{T_{1}}=\begin{pmatrix}
    u_{x_{i}x_{j}} & u_{x_{i}y_{j}}  \\
     0 & I_{m}
\end{pmatrix}=\begin{pmatrix}
    \Phi^{-1} & -\Phi^{-1}\Gamma  \\
    0 & I_{m} \end{pmatrix}, \quad \label{JacT1} (J_{T_{1}})^{-1}=\begin{pmatrix}
    \Phi & \Gamma \\
     0 & I_{m}  
\end{pmatrix},\end{align}
with the matrices $\Phi,\Gamma$ from Theorem \ref{CH} and Corollary \ref{MAP},
\begin{align}\label{ph}\Phi=\frac{1}{4}(\Lambda_{\tilde{x}_{i}\tilde{x}_{j}}) = -(\Lambda_{\tilde{z}_{i}\overline{\tilde{z}}_{j}}), \quad \tilde{z}_i := z_i, \quad  \Gamma=\frac{1}{4}(\Lambda_{\tilde{x}_{i}\tilde{y}_{j}}) = \Gamma^\top,\end{align}
which allow us to write 
$${\rm Hess}(u) = {\rm Hess}_{\mathbb{C}}(K) = \begin{pmatrix}
   \Phi^{-1}  & -\Phi^{-1}\Gamma   \\
   - {\Gamma}\Phi^{-1}  & \Phi +{\Gamma} \Phi^{-1} \Gamma
\end{pmatrix}.$$
Also, because $\Phi^{-1}$ is positive definite and $B_{\delta(O)}(O)$ is convex in the $x_i$-directions, $T_1$ is injective.

\begin{lemma}\label{alpha}
    In our setting, $\Lambda$ is pluriharmonic in the coordinates $\alpha_i := \tilde{x}_i + \sqrt{-1}\tilde{y}_i$.
\end{lemma}

\begin{proof}
Notice that for all $k,\ell=1,\ldots,m$,
\begin{align*}
\frac{\partial^{2}\Lambda}{\partial \alpha_{k}\partial\overline{\alpha}_{\ell}} & =  \displaystyle\frac{1}{4}\left( \frac{\partial}{\partial \tilde{x}_{k}}-\sqrt{-1}\frac{\partial}{\partial \tilde{y}_{k}} \right)\left( \frac{\partial \Lambda}{\partial \tilde{x}_{\ell}}+\sqrt{-1}\frac{\partial \Lambda}{\partial \tilde{y}_{\ell}} \right)\\
& =  \displaystyle \frac{1}{4} \left(\frac{\partial^2 \Lambda}{\partial\tilde{x}_{k}\partial \tilde{x}_{\ell}}+\frac{\partial^2 \Lambda}{\partial \tilde{y}_{k}\partial \tilde{y}_{\ell}}\right)+\frac{\sqrt{-1}}{4}\left(\frac{\partial^2 \Lambda}{\partial \tilde{x}_{k}\partial \tilde{y}_{\ell}}-\frac{\partial^2 \Lambda}{\partial  \tilde{y}_{k} \partial\tilde{x}_{\ell}}\right). 
\end{align*}
This vanishes by \eqref{ph}.
\end{proof}

\begin{lemma}\label{Zholo}
    In our setting, $Z := \Gamma+\sqrt{-1}\Phi$ is holomorphic in the coordinates $\alpha_i$.
\end{lemma}

\begin{proof}
    By \eqref{ph}, we have that
$$\frac{\partial}{\partial \tilde{x}_{k}}\Gamma_{ij}= \displaystyle\frac{1}{4}\Lambda_{\tilde{x}_{i}\tilde{y}_{j}\tilde{x}_{k}}
     = -\displaystyle\frac{1}{4}\Lambda_{\tilde{y}_{i}\tilde{y}_{j}\tilde{y}_{k}} 
    = \displaystyle\frac{1}{4}\Lambda_{\tilde{x}_{i}\tilde{x}_{j} \tilde{y}_{k}}
    = \displaystyle \frac{\partial}{\partial \tilde{y}_{k}} \Phi_{ij},
$$
$$\frac{\partial}{\partial \tilde{y}_{k}}\Gamma_{ij}  = \displaystyle\frac{1}{4}\Lambda_{\tilde{x}_{i}\tilde{y}_{j}\tilde{y}_{k}}
     = - \displaystyle\frac{1}{4}\Lambda_{\tilde{x}_{i}\tilde{x}_{j}\tilde{x}_{k}}
     = - \displaystyle \frac{\partial}{\partial \tilde{x}_{k}} \Phi_{ij}.
$$
Thus, every component function $Z_{ij} = \Gamma_{ij} +\sqrt{-1}\Phi_{ij}$ is holomorphic with respect to $\alpha_k$. 
\end{proof}

\begin{definition}\label{def:Siegel}
The \emph{Siegel upper half-plane} and the \emph{Siegel disk} are defined by
\begin{align*}\mathfrak{H}_{m}&:=\{Z\in \C^{m\times m}: Z=Z^{\top},\ \mathrm{Im}(Z)>0\},\\
\mathbb{D}_m &:=\{W\in \C^{m\times m}: W=W^{\top} ,\ I_{m}-W\overline{W}^{\top}>0 \}.
\end{align*}
The \emph{Cayley transformation} is the biholomorphism
$$\mathscr{C}:\mathfrak{H}_{m}\to \mathbb{D}_{m}, \quad Z \mapsto (Z-\sqrt{-1}I_{m})(Z+\sqrt{-1}I_{m})^{-1}.$$
\end{definition}

\begin{remark}\label{bounded}  
The Siegel disk $\mathbb{D}_m$ is a bounded domain: for all $W\in \mathbb{D}_{m}$,
$$\|W\|_F = \sqrt{\mathrm{tr}(W\overline{W}^{\top})} < \sqrt{m}.$$
\end{remark}
    
\begin{remark}
The same procedure can also be carried out by fixing the set of holomorphic Killing vector fields $\partial_{t_i}$, $i = 1,\ldots, m$. Thus, we construct a function $\Lambda^{*}/4$ as the Legendre transform of $K/4$ with respect to the coordinates $y_{i}$, and $\Lambda^*/4$ is defined in the new coordinates
$$\hat{x}_{i}:=x_{i},\quad  \hat{y}_{i}:=\frac{K_{{\rm Re}(z_i)}}{4} = \frac{K_{z_{i}}}{2}=u_{y_{i}}.$$
These new coordinates are given by the diffeomorphism
\begin{align}\label{mapT2}T_{2}: B_{\delta(O)}(O)\to T_{2}(B_{\delta(O)}(O)), \quad (\hat{x}_i, \hat{y}_i) = T_2(x_i,y_i) := (x_{i},u_{y_{i}}).\end{align}
We obtain symmetric real matrices
$$\Phi^* := -\frac{1}{4}(\Lambda^{*}_{\hat{x}_{i}\hat{y}_{j}}),\quad  \Gamma^* :=\frac{1}{4}(\Lambda^{*}_{\hat{y}_{i}\hat{y}_{j}}).$$
Then $\Lambda^{*}$ is pluriharmonic with respect to the coordinates $\beta_{i}:=\hat{x}_{i}+\sqrt{-1}\hat{y}_{i}$ on the image of $T_2$, and the $\mathfrak{H}_m$-valued map $Z^*:=\Phi^*+\sqrt{-1}\Gamma^*$ is holomorphic with respect to these coordinates.
\end{remark}

We now introduce a third change of variables, which will play a central role in the remainder of this section. Recalling \eqref{mapT1} and \eqref{mapT2}, define $T_{3}: B_{\delta(O)}(O)\to T_{3}(B_{\delta(O)}(O))$ by
\begin{equation}\label{T3m}
     T_{3}(x_{i},y_{i}) := T_{1}(x_{i},y_{i})+T_{2}(x_{i},y_{i})= (u_{x_{i}},y_{i})+(x_{i},u_{y_{i}})= (x_{i}+u_{x_{i}},y_{i}+u_{y_{i}}).
\end{equation}
This map has Jacobian
\begin{align}\label{JacT3}J_{T_{3}}=\left(\begin{array}{cc}
   I_{m}+(u_{x_{i}x_{j}})  & (u_{x_{i}y_{j}})  \\
    (u_{y_{i}x_{j}}) & I_{m}+(u_{y_{i}y_{j}})
\end{array}\right) = I_{2m} + {\rm Hess}(u).\end{align}
As before, we conclude that $T_3$ is a diffeomorphism onto its image. On the image of $T_3$ we introduce the associated complex coordinates $\sigma_i$, i.e.
\begin{align}\label{def_sigma}
\sigma_{i} \circ T_3:=(x_{i}+u_{x_{i}})+\sqrt{-1}(y_{i}+u_{y_{i}}).
 \end{align}
    For $m=1$ we recover the coordinate $\sigma$ introduced by Lewy \cite{Lewy1,Lewy2} and Nitsche \cite{Nit} in their studies of the real Monge-Ampère equation on $\mathbb{R}^2$.

\begin{lemma}\label{compose}
    The composition $T_{3}\circ T^{-1}_{1}$ is a biholomorphism from $T_1(B_{\delta(O)}(O)) \subset \mathbb{R}^{2m}$, equipped with the $\alpha$-coordinates, onto $T_3(B_{\delta(O)}(O)) \subset \mathbb{R}^{2m}$, equipped with the $\sigma$-coordinates.
\end{lemma}
    
\begin{proof}
By definition this means that the Jacobian $J_{T_3 \circ T_1^{-1}}$ commutes with
$$\begin{pmatrix}
    0 & -I_{m}  \\
     I_{m} & 0
\end{pmatrix},$$
the standard complex structure of $\mathbb{R}^{2m}$. Using \eqref{JacT1} and \eqref{JacT3} we compute
\begin{align*}
  J_{T_{3}\circ T^{-1}_{1}}    =  J_{T_{3}}\circ (J_{T_{1}})^{-1}
      = \begin{pmatrix}
        I_{m}+\Phi^{-1}  & -\Phi^{-1}\Gamma  \\
        -\Gamma\Phi^{-1}  & I_m+\Phi+\Gamma\Phi^{-1}\Gamma 
    \end{pmatrix}\begin{pmatrix}
         \Phi & \Gamma  \\
          0 & I_{m}
     \end{pmatrix}
     &= \begin{pmatrix}
         I_{m}+\Phi & \Gamma  \\
          -\Gamma & I_{m}+\Phi
   \end{pmatrix}.
     \end{align*}
Then the desired property is easy to check.
\end{proof}

\begin{lemma}\label{Holchart1}
    The vector-valued map
    \begin{align}\label{def_f}f:T_{3}(B_{\delta(O)}(O))\to \C^{m}, \quad f_{i} \circ T_3 :=(x_{i}-u_{x_{i}})-\sqrt{-1}(y_{i}-u_{y_{i}}),\end{align}
    is holomorphic with respect to the coordinates $\sigma_i$. Moreover,
    $$\frac{\partial f}{\partial \sigma}=\mathscr{C} \circ Z \circ T_1 \circ T_3^{-1},$$
    where $Z: T_1(B_{\delta(O)}(O)) \to \mathfrak{H}_m$ is as in Lemma \ref{Zholo} and $\mathscr{C}: \mathfrak{H}_m \to \mathbb{D}_m$ is the Cayley transformation.
\end{lemma}

This again generalizes a result of Nitsche \cite{Nit} to higher dimensions. For $m = 1$ Nitsche proved by a direct computation that $|f'(\sigma)| < 1$, which is the key to his proof of Bernstein's theorem on $\mathbb{R}^2$.

\begin{proof}
    Recall the $\alpha$-coordinates from Lemma \ref{alpha}. By definition,
    $$f(\sigma)=\sigma-2\alpha.$$
    Thus, by Lemma \ref{compose}, $f$ is holomorphic with respect to $\sigma$. We will now abuse notation by identifying $Z=\Gamma+\sqrt{-1}\Phi$ with $Z \circ T_1 \circ T_3^{-1}$. Then, from the proof of Lemma \ref{compose},
    \begin{align*}\frac{\partial f}{\partial \sigma}=I_{m}-2\frac{\partial \alpha}{\partial \sigma} &= I_m - 2(\Phi+I_{m}-\sqrt{-1}\Gamma)^{-1}\\
         & =  (\Phi+I_{m}-\sqrt{-1}\Gamma-2I_{m})(\Phi+I_{m}-\sqrt{-1}\Gamma)^{-1}\\
         & =  ((\Gamma+\sqrt{-1}\Phi)-\sqrt{-1}I_{m})((\Gamma+\sqrt{-1}\Phi)+\sqrt{-1}I_{m})^{-1}\\
         & =  (Z-\sqrt{-1}I_{m})(Z+\sqrt{-1}I_{m})^{-1},
    \end{align*}
    which is by definition equal to $\mathscr{C} \circ Z$.
\end{proof}

\subsection{Semi-flat hyper-Kähler metrics: estimates and main result}

We now turn to the proof of Theorem \ref{thm:main}. We start by recalling a classical lemma which is implicit in \cite{Nit}.

\begin{lemma}\label{DT}
Let $u:D\subset \R^{n}\to \R$ be a smooth convex function on a convex domain. Define 
$$T: D \to \mathbb{R}^n, \quad T(p) := p + (\nabla u)(p)\;\,\text{for all}\;\,p \in D.$$ Then $T$ is a diffeomorphism onto its image. Moreover, if $O\in D$ and if $\nabla u(O)=0$, then for every Euclidean ball $B_{r}(O)\subset D$ we have that $B_{r}(O)\subset T(B_r(O))$.
\end{lemma}

The following lemma, which is new even for $m = 1$, explains why we are able to improve Nitsche's estimate $M_2 \leq 4$ from \cite{Nit} to $M_2 \leq 3$ in this case. Note that Lemma \ref{DT} implies $\delta(O) \leq R$.

\begin{lemma}\label{3NN}
Let $u:U\subset\R^{2m}\to\R$ be smooth and convex with symplectic Hessian such that \eqref{symplectic} holds. As before, let $\delta(O)$ be the maximal radius such that $B_{\delta(O)}(O)\subset U$. Consider the map $$T_{3}:B_{\delta(O)}(O)\to T_{3}(B_{\delta(O)}(O))$$
from \eqref{T3m}. Let $R>0$ be the maximal radius such that $B_R(O)\subset T_{3}(B_{\delta(O)}(O))$. Then
$$\frac{R}{4}\leq \delta(O)\leq \frac{3R}{4}.$$
\end{lemma}

\begin{proof}
By Lemma \ref{Holchart1} and \eqref{symplectic} and by the definition of $\mathscr{C}$, we have that 
\begin{equation}\label{op}
    f(O) = 0, \quad \frac{\partial f}{\partial \sigma}(O) = \mathscr{C}(\sqrt{-1}I_m) = 0, \quad I_{m}-\frac{\partial f}{\partial \sigma}\,\overline{\frac{\partial f}{\partial \sigma}}^{\top}>0.
\end{equation}

\noindent {\bf Claim:} Let $\|\cdot\|_{\operatorname{op}}$ denote the operator norm on $\C^{m\times m}$ defined by 
$$\|A\|_{\operatorname{op}}:=\sup_{|v|=1} |Av|,$$ 
where $|\cdot|$ is the standard Euclidean norm on $\C^{m}$. Then we have that 
$$\left\| \frac{\partial f}{\partial \sigma}(\sigma)\right\|_{\operatorname{op}}\leq \frac{|\sigma|}{R}\;\,\text{for all}\;\,\sigma \in B_R(O).$$

\noindent {\bf Proof of the Claim:}
The Claim is trivial from \eqref{op} at $\sigma = O$. Thus, fix any $\sigma\in B_{R}(O)\setminus\{O\} $. For any two unit vectors $v,w\in \C^{m}$, define an auxiliary function $\phi_{v,w}:B_{R}(0)\subset\C\to \C$ by
$$\phi_{v,w}(\zeta):= \left\langle \frac{\partial f}{\partial \sigma}\left(\zeta \frac{\sigma}{|\sigma|}\right) \cdot v,w \right\rangle,$$
where $\langle \cdot, \cdot \rangle$ is the standard Hermitian inner product on $\mathbb{C}^m$ (complex linear in the first variable).
This is holomorphic because $f$ is. Moreover, $\phi_{v,w}(0) = 0$ and $|\phi_{v,w}(\zeta)| < 1$ for all $\zeta \in B_R(0)$ because \eqref{op} implies that all of the singular values of $\partial f/\partial \sigma$ are $< 1$, so that $\|\partial f/\partial \sigma\|_{\rm op} < 1$. Thus, the Schwarz lemma yields $|\phi_{v,w}(\zeta)| \leq |\zeta|/R$. The Claim follows by choosing $\zeta = |\sigma|$ and maximizing over $v,w$. \hfill $\Box$\medskip

\noindent To proceed, note that since $f(O)=0$,
$$
  |f(\sigma)| = \displaystyle \left|\int_{0}^{1} \frac{\partial f}{\partial \sigma}(t\sigma)\sigma \, dt\right|  
           \leq  \displaystyle \int_{0}^{1} \left\|\frac{\partial f}{\partial \sigma}(t\sigma)\right\|_{\operatorname{op}}|\sigma| \, dt 
           \leq  \displaystyle  \int_{0}^{1} t\frac{|\sigma|}{R}|\sigma| \, dt 
           =  \displaystyle  \frac{|\sigma|^2}{2R} 
           <   \displaystyle \frac{R}{2}
$$
if $|\sigma|<R$. Therefore, by using the equation
\begin{align*}
    T_3^{-1}(\sigma) = x+\sqrt{-1}y=\frac{\sigma+\overline{f(\sigma)}}{2},
\end{align*}
which is obvious from \eqref{def_sigma}, \eqref{def_f}, we conclude that
\begin{equation}\label{xy-sigma}
\frac{|\sigma|}{2}-\frac{R}{4} < |T_3^{-1}(\sigma)| < \frac{|\sigma|}{2}+\frac{R}{4} \;\,\text{for all}\;\,\sigma \in B_R(O).
\end{equation}

\noindent {\bf Proof of the inequality $R/4\leq\delta(O)$:} Let $0<\varepsilon<R$, and consider $B_{R-\varepsilon}(O)\subset T_{3}(B_{\delta(O)}(O))$. Set $$U_{\varepsilon}:=T_{3}^{-1}(B_{R-\varepsilon}(O))\subset B_{\delta(O)}(O).$$
Then $U_\varepsilon$ contains the point $O = T_3^{-1}(O)$, and $U_\varepsilon$ is an open set because $T_3$ is continuous. Moreover, $\overline{U_{\varepsilon}}\subset B_{\delta(O)}(O)$ as well because for all $u \in \partial U_\varepsilon$ we can pick a sequence $u_i \in U_\varepsilon$ converging to $u$; then $v_i := T_3(u_i) \in B_{R-\varepsilon}(O)$, so the sequence $v_i$ has a limit $v \in \overline{B_{R-\varepsilon}(O)} \subset B_R(O) \subset T_3(B_{\delta(O)}(O))$ after passing to a subsequence, so $u = T_3^{-1}(v) \in B_{\delta(O)}(O)$ by the continuity of $T_3^{-1}$.

Let $\delta_{\varepsilon}(O)$ denote the Euclidean distance from $O$ to $\partial U_{\varepsilon}$. Since $\overline{U_\varepsilon}\subset B_{\delta(O)}(O)$, there exists $z_\varepsilon\in\partial U_\varepsilon $ such that $|z_{\varepsilon}|=\delta_{\varepsilon}(O) < \delta(O)$. If  $\sigma_{\varepsilon}:=T_{3}(z_{\varepsilon})$, then $\sigma_\varepsilon\in \overline{B_{R-\varepsilon}(O)}$ as above. In fact, $\sigma_\varepsilon \in \partial B_{R-\varepsilon}(O)$ because otherwise $z_\varepsilon \in T_3^{-1}(B_{R-\varepsilon}(O)) = U_\varepsilon$. Thus, by \eqref{xy-sigma} we have that
$$\frac{R-\varepsilon}{2} - \frac{R}{4} = \frac{|\sigma_{\varepsilon}|}{2}-\frac{R}{4} < |z_{\varepsilon}| = \delta_\varepsilon(O) < \delta(O).$$
Letting $\varepsilon\to 0$, we conclude the claim.\medskip

\noindent {\bf Proof of the inequality $\delta(O)\leq 3R/4$:} Let $0<\varepsilon<\delta(O)$, and consider $B_{\delta(O)-\varepsilon}(O)\subset B_{\delta(O)}(O)$. Set $$V_{\varepsilon}:=T_{3}(B_{\delta(O)-\varepsilon}(O))\subset T_{3}(B_{\delta(O)}(O)).$$
Then $V_\varepsilon$ contains the point $O = T_3(O)$, and $V_\varepsilon$ is an open set because $T_3^{-1}$ is continuous. Moreover, $\overline{V_{\varepsilon}}\subset T_3(B_{\delta(O)}(O))$ as well because for all $v \in \partial V_\varepsilon$ we can pick a sequence $v_i \in V_\varepsilon$ converging to $v$; then $u_i := T_3^{-1}(v_i) \in B_{\delta(O)-\varepsilon}(O)$, so the sequence $u_i$ has a limit $u \in \overline{B_{\delta(O)-\varepsilon}(O)} \subset B_{\delta(O)}(O)$ after passing to a subsequence, so $v = T_3(u) \in T_3(B_{\delta(O)}(O))$ by the continuity of $T_3$.

Let $R_{\varepsilon}$ denote the Euclidean distance from $O$ to $\partial V_{\varepsilon}$. Since $\overline{V_{\varepsilon}}\subset T_{3}(B_{\delta(O)}(O))$, there exists $\sigma_{\varepsilon}\in \partial V_{\varepsilon}$ such that $|\sigma_{\varepsilon}|=R_{\varepsilon} < R$. If $z_{\varepsilon}:=T_{3}^{-1}(\sigma_{\varepsilon})$, then $z_\varepsilon\in \overline{B_{\delta(O)-\varepsilon}(O)}$ as above. In fact, $z_\varepsilon \in \partial B_{\delta(O)-\varepsilon}(O)$ because otherwise $\sigma_\varepsilon \in T_3(B_{R-\varepsilon}(O)) = V_\varepsilon$. Thus, by \eqref{xy-sigma} we have that
$$\delta(O)-\varepsilon = |z_{\varepsilon}|< \frac{|\sigma_{\varepsilon}|}{2}+\frac{R}{4} = \frac{R_\varepsilon}{2}+\frac{R}{4} < \frac{3R}{4}.$$
Letting $\varepsilon\to 0$, we conclude the claim.
\end{proof}

The following lemmas are again generalizations of computations due to Nitsche \cite{Nit} for $m = 1$.

\begin{lemma}\label{Sim}
Let $u:U\subset\R^{2m}\to\R$ be smooth and convex with symplectic Hessian such that \eqref{symplectic} holds. Then the scalar curvature $R_g$ of the associated Hessian metric $g$ on $U$ satisfies
\begin{equation}\label{SCO}R_g(O)=\sum_{i,j,k=1}^{m} (u_{x_{i}x_{j}x_{k}}(O)^{2}+u_{x_{i}x_{j}y_{k}}(O)^{2}).
\end{equation}
\end{lemma}

\begin{proof}
By Lemma \ref{SCT} and by the normalization ${\rm Hess}(u)|_O = I_{2m}$,
\begin{align}\label{S}
\begin{split}
    R_g(O)&= \frac{1}{4}\sum_{i,j,k=1}^{m} (u_{x_{i}x_{j}x_{k}}(O)^{2} + u_{x_{i}y_{j}x_{k}}(O)^{2} +u_{x_{i}x_{j}y_{k}}(O)^{2} +u_{x_{i}y_{j}y_{k}}(O)^{2} \\
    &   \displaystyle \ \ \ \ \ \ \hspace{5.5mm}\ +\,u_{y_{i}x_{j}x_{k}}(O)^{2} +u_{y_{i}y_{j}x_{k}}(O)^{2}+u_{y_{i}x_{j}y_{k}}(O)^{2}+u_{y_{i}y_{j}y_{k}}(O)^{2}).
\end{split}
\end{align}
Since ${\rm Hess}(u)\in {\rm Sp}(2m,\R)$, by equation \eqref{SP2} we obtain
$$\sum_{\ell=1}^m (u_{x_i x_\ell} u_{y_\ell y_j} - u_{x_i y_\ell} u_{x_\ell y_j})= \delta_{ij} \quad \text{for all} \  i,j = 1, \ldots, m.$$
Differentiating with respect to $x_{k}$ yields
$$\sum_{\ell=1}^m (u_{x_i x_\ell x_k} u_{y_\ell y_j} + u_{x_i x_\ell} u_{y_\ell y_j x_k} - u_{x_i y_\ell x_k} u_{x_\ell y_j} - u_{x_i y_\ell} u_{x_\ell y_j x_k}) = 0.$$
Recall that $u_{x_{i}x_{j}}(O)=u_{y_{i}y_{j}}(O)=\delta_{ij}$ and $u_{x_{i}y_{j}}(O)=0$. Thus,
$$u_{x_{i}x_{j}x_{k}}(O) =-u_{y_{i}y_{j}x_{k}}(O) \quad \text{for all}\  i,j,k=1,\ldots,m.$$
Similarly, differentiating with respect to $y_{k}$ we find that
$$u_{x_i x_j y_k}(O) = - u_{y_i y_j y_k}(O) \quad \text{for all}\ i,j,k=1,\ldots,m.$$
Substituting these into \eqref{S}, we obtain \eqref{SCO}. 
\end{proof}

\begin{lemma}\label{CT}
For all $Z \in \mathfrak{H}_m$ we have that
$$\mathrm{Im}(Z)^{-1}=(I_{m}-\overline{{\mathscr{C}(Z)}})(I_{m}-\mathscr{C}(Z)\overline{\mathscr{C}(Z)})^{-1}(I_{m}-\mathscr{C}(Z)).$$
\end{lemma}

\begin{proof}
If $W := \mathscr{C}(Z)$, then, by inverting the Cayley transformation,
$$Z =\sqrt{-1}(I_{m}+W)(I_{m}-W)^{-1},$$
and by complex conjugation we get
$$\overline{Z}=-\sqrt{-1}(I_{m}+\overline{W})(I_{m}-\overline{W})^{-1}.$$
Hence, observing that $I_m +W$ and $(I_m-W)^{-1}$ commute,
\begin{align*}
  \mathrm{Im}(Z) & =  \displaystyle \frac{1}{2\sqrt{-1}}(Z-\overline{Z})  \\
& =  \displaystyle \frac{1}{2\sqrt{-1}}(\sqrt{-1}(I_{m}-W)^{-1}(I_{m}+W)+\sqrt{-1}(I_m+\overline{W})(I_{m}-\overline{W})^{-1})\\
& = \displaystyle \frac{1}{2}(I_{m}-W)^{-1}((I_{m}+W)(I_{m}-\overline{W}) +(I_{m}-W)(I_{m}+\overline{W}))(I_{m}-\overline{W})^{-1}\\
   & =  (I_{m}-W)^{-1}(I-W\overline{W})(I_{m}-\overline{W})^{-1}.
\end{align*}
The lemma follows from this by inversion.
\end{proof}

\begin{proof}[Proof of Theorem \ref{thm:main}]
Fix an arbitrary point $x \in U$, aiming to prove that the main estimate \eqref{eq:main_est} is true at the point $x$. We may assume that $x$ is the origin $O \in \mathbb{R}^{2m}$ and that the normalization \eqref{symplectic} holds. By Lemma \ref{Sim} and because $\Phi^{-1} = (u_{x_ix_j})$, it holds at the point $O$ that
$$R_g=\sum_{i,j,k=1}^{m} (u_{x_{i}x_{j}x_{k}}^{2}+u_{x_{i}x_{j}y_{k}}^{2}) = \sum_{k=1}^{m} (\|\partial_{x_k}\Phi^{-1}\|_F^2 + \|\partial_{y_k}\Phi^{-1}\|_F^2).$$
By Lemma \ref{CT},
$$\Phi^{-1}=(I_{m}-\overline{\mathscr{C}(Z)})(I_{m}-\mathscr{C}(Z)\overline{\mathscr{C}(Z)})^{-1}(I_{m}-\mathscr{C}(Z)).$$
Since $Z(O) = \sqrt{-1}I_m$ by \eqref{symplectic}, and hence $\mathscr{C}(Z)(O) = 0$, we have 
$$\partial_{x_k}\Phi^{-1} = -2{\rm Re}(\partial_{x_k}\mathscr{C}(Z)),\quad  \partial_{y_k}\Phi^{-1} = -2{\rm Re}(\partial_{y_k} \mathscr{C}(Z)),$$
again at $O$. By the chain rule, because $\mathscr{C}(Z)$ is holomorphic with respect to $\sigma$, and by \eqref{def_sigma},
$$\partial_{x_k}\mathscr{C}(Z) = \sum_{i=1}^m \partial_{\sigma_i}\mathscr{C}(Z) \cdot \partial_{x_k}\sigma_i = 2\partial_{\sigma_k}\mathscr{C}(Z),\quad  \partial_{y_k}\mathscr{C}(Z) = 2\sqrt{-1}\partial_{\sigma_k}\mathscr{C}(Z)$$
at $O$. Thus, putting everything together,
\begin{equation}\label{f''}
    R_g(O)= 16 \sum_{k=1}^m \|\partial_{\sigma_k}\mathscr{C}(Z)(O)\|_F^2.
\end{equation}

Now let $R > 0$ denote the maximal radius such that $B_R(O)\subset T_{3}(B_{\delta(O)}(O))$ as in Lemma \ref{3NN}, and recall that $\delta(O)$ is by definition equal to the quantity ${\rm dist}|_{g_x}(x,\partial U)$ in the desired estimate \eqref{eq:main_est}.
Thus, for every $k \in \{1,\ldots,m\}$, the domain $T_3(B_{\delta(O)}(O))$ contains the complex $1$-dimensional disk of radius $R$ centered at the origin inside the $\sigma_k$-axis in $\mathbb{C}^m$. Call this disk $\Delta_{k}$.

The standard proof of the Cauchy inequality for holomorphic functions on $\Delta_k$ applies verbatim to holomorphic maps with values in a finite-dimensional normed complex vector space; all that is needed for this is the triangle inequality for integrals. Thus,
$$\|\partial_{\sigma_k} \mathscr{C}(Z)(O)\|_F \leq \frac{1}{R}\sup_{\Delta_k} \|\mathscr{C}(Z)\|_F \leq \frac{\sqrt{m}}{R}$$
by Remark \ref{bounded}. Thus,
$$R_g(O)\leq 16 \sum_{k=1}^{m}\frac{m}{R^2}=\frac{16m^2}{R^2}.$$
By combining this with Lemma \ref{3NN}, we obtain the desired estimate \eqref{eq:main_est}.
\end{proof}

\subsection{Inversion of the holomorphic description}\label{2.5}

One might worry that except for the case $m = 1$ and direct sums there are no examples of convex functions with symplectic Hessian. To fix this gap, and to be able to study some concrete examples in Section \ref{sec:special}, we now invert our construction.

\begin{proposition}\label{MACorr}
There exist constructive local correspondences between the following objects:
\begin{itemize}
    \item[\textup{(1)}] A convex function $u:U \subset \R^{2m} \to\R$ satisfying $\det\hspace{0.25mm} {\rm Hess}(u) = 1$ and ${\rm Hess}(u)\in{\rm Sp}(2m,\R)$ with respect to the standard real coordinates $(x_i,y_i)$ of $\mathbb{R}^{2m}$.
    \item[\textup{(2)}] A function $\Lambda:\widetilde{U} \subset \C^m \to\R$, pluriharmonic with respect to the standard complex coordinates $\alpha_{i}=\tilde{x}_{i}+\sqrt{-1}\tilde{y}_{i}$ of $\mathbb{C}^m$, such that $(\Lambda_{\tilde{x}_{i}\tilde{x}_{j}}) \in \mathbb{R}^{m\times m}$ is positive definite.
    \item[\textup{(3)}] A function $\mathfrak{h}:\widetilde{U} \subset \C^m \to\C$, holomorphic with respect to the standard complex coordinates $\alpha_{i}=\tilde{x}_{i}+\sqrt{-1}\tilde{y}_{i}$ of $\mathbb{C}^m$, such that ${\rm Im}(\mathfrak{h}_{\alpha_i\alpha_j}) \in \R^{m\times m}$ is positive definite.
    \item[\textup{(4)}] A function $Z = \Gamma + \sqrt{-1}\Phi:\widetilde{U} \subset \C^m \to\mathfrak{H}_m$, holomorphic with respect to the standard complex coordinates $\alpha_{i}=\tilde{x}_{i}+\sqrt{-1}\tilde{y}_{i}$ of $\mathbb{C}^m$, such that $Z_{ij,\alpha_k}=Z_{ik,\alpha_j}$ for all $i,j,k = 1,\ldots,m$.
\end{itemize}
\end{proposition}

\begin{remark}
The above can be extrapolated from the literature but to our knowledge it was never stated explicitly. (3) $\longleftrightarrow$ (4) can be found in \cite{Freed}, where $Z$ is used to construct a special Kähler structure on $\widetilde{U}$. Special Kähler manifolds have associated semi-flat hyper-Kähler structures on their cotangent bundles. In \cite{Lu} it was proved that complete special Kähler manifolds are flat, using an argument very similar to Calabi's scalar curvature estimate \cite{EC} (see Proposition \ref{prop:DiffIneqR} below). Later, \cite{ACD,BC,VC} found a direct link between special Kähler structures and improper affine hyperspheres, hence with (1).
\end{remark}

\begin{proof}[Proof of Proposition \ref{MACorr}]
    (3) $\longrightarrow$ (2) is setting $\Lambda := {\rm Im}(4\mathfrak{h})$, and (3) $\longrightarrow$ (4) is setting $Z := (\mathfrak{h}_{\alpha_i \alpha_j})$. The converses are obtained from the Poincaré lemma. Locally a pluriharmonic function $\Lambda$ always has a conjugate (a primitive of the closed $1$-form $d^c\Lambda$), hence is the imaginary part of a holomorphic function $4\mathfrak{h}$. The holomorphic Poincaré lemma allows us to first go from a $Z$ as in (4) to a holomorphic vector $(h_1,\ldots, h_m)$ such that $Z_{ij} = h_{i,\alpha_j}$, and then to a holomorphic function $\mathfrak{h}$ such that $h_i = \mathfrak{h}_{\alpha_i}$. 

    (1) $\longrightarrow$ (2) is our previous work: By Lemma \ref{MAMG}, $M=U\times \R^{2m}$ with real coordinates $(x_{i},y_{i},s_{i},t_{i})$ is a Calabi-Yau manifold with holomorphic coordinates $(w_{i},z_{i})=(x_{i}+\sqrt{-1}s_{i},y_{i}+\sqrt{-1}t_{i})$, Kähler form $\omega_M$ given by the Kähler potential $K(w_{i},z_{i})=4u(x_{i},y_{i})$, and $2m$ commuting and pointwise linearly independent holomorphic Killing vector fields $\partial_{s_i}, \partial_{t_i}$. By Theorem \ref{CH} and Corollary \ref{MAP}, the partial Legendre transform $\Lambda/4$ of $K/4$ with respect to $x_1,\ldots,x_m$, defined in the coordinates
$$\tilde{x}_{i}:=\frac{K_{{\rm Re}(w_i)}}{4} = \frac{K_{w_{i}}}{2}=u_{x_{i}} ,\quad  \tilde{y}_{i}:=y_{i},$$
allows us to write the complex Hessian of $K$ as
$${\rm Hess}_{\C}(K)=\begin{pmatrix}
    \Phi^{-1} & -\Phi^{-1}\Gamma  \\
     -\Gamma\Phi^{-1}  & \Phi+\Gamma \Phi^{-1}\Gamma 
\end{pmatrix}, \quad \Phi:=\frac{1}{4}(\Lambda_{\tilde{x}_{i}\tilde{x}_{j}}) = -\frac{1}{4}(\Lambda_{\tilde{y}_{i}\tilde{y}_{j}}),\quad  \Gamma:=\frac{1}{4}(\Lambda_{\tilde{x}_{i}\tilde{y}_{j}})=\Gamma^\top.$$
By Lemma \ref{alpha}, $\Lambda$ is pluriharmonic in the coordinates $\alpha_{i}:=\tilde{x}_{i}+\sqrt{-1}\tilde{y}_{i}$. On the other hand, since $K$ is a Kähler potential, we must have that $\Phi^{-1}>0$.

(2) $\longrightarrow$ (1): Taking the partial Legendre transform of $\Lambda/4$ with respect to $\tilde{x}_{i}$, we obtain
$$u(x_i,y_i):=\frac{1}{4}K(x_i,y_i):=-\frac{1}{4}\Lambda(\tilde{x}_i,\tilde{y}_i)+\sum_{i=1}^{m} \tilde{x}_{i}x_{i}, \quad x_{i}:=\frac{\Lambda_{\tilde{x}_{i}}}{4} ,\quad  y_{i}:=\tilde{y}_{i}.$$
Then, by a calculation similar to \eqref{KF},
\begin{align*}
   u_{x_i x_j} =  4\Lambda^{\tilde{x}_{i}\tilde{x}_{j}}, \quad 
   u_{y_i x_j}  =  -\Lambda_{\tilde{y}_i \tilde{x}_k}\Lambda^{\tilde{x}_{k}\tilde{x}_{j}}, \quad
   u_{y_i y_j}  =  \displaystyle  -\frac{1}{4}\Lambda_{\tilde{y}_i\tilde{y}_j}+\displaystyle \frac{1}{4} \Lambda_{\tilde{y}_{i}\tilde{x}_{k}}\Lambda^{\tilde{x}_{k}\tilde{x}_{\ell}}\Lambda_{\tilde{x}_{\ell}{\tilde{y}}_{j}}.
\end{align*}
Since $\Lambda$ is $\alpha$-pluriharmonic, we have that $\Lambda_{\tilde{x}_{i}\tilde{x}_{j}}=-\Lambda_{\tilde{y}_{i}\tilde{y}_{j}}$ and $\Lambda_{\tilde{x}_{i}\tilde{y}_{j}}=\Lambda_{\tilde{y}_{i}\tilde{x}_{j}}$. So, by setting
   $$\Phi:=\frac{1}{4}(\Lambda_{\tilde{x}_{i}\tilde{x}_{j}}) , \ \ \Gamma:=\frac{1}{4}(\Lambda_{\tilde{x}_{i}\tilde{y}_{j}}),$$
we can express the Hessian of $u$ as
   $$
     {\rm Hess}(u)  = \begin{pmatrix}
     (u_{x_{i}x_{j}})  & (u_{x_{i}y_{j}})  \\
     (u_{y_{i}x_{j}})  & (u_{y_{i}y_{j}})
     \end{pmatrix} 
    = \begin{pmatrix}
    \Phi^{-1}  & -\Phi^{-1}\Gamma   \\
    -\Gamma\Phi^{-1}  & \Phi +\Gamma \Phi^{-1} \Gamma
\end{pmatrix}. $$
   Thus, ${\rm Hess}(u)\in {\rm Sp}(2m,\R)$. Since $(\Lambda_{\tilde{x}_{i}\tilde{x}_{j}})$ is positive definite, the upper left block $(u_{x_{i}x_{j}})=\Phi^{-1}$ and its Schur complement $S=(u_{y_{i}y_{j}})-(u_{y_{i}x_{j}})(u_{x_{i}x_{j}})^{-1}(u_{x_{i}y_{j}})=\Phi$ are positive definite as well. This implies that ${\rm Hess}(u)$ is positive definite, so $u$ is convex. Lastly, we recall that ${\rm Sp}(2m,\mathbb{R}) \subset {\rm SL}(2m,\mathbb{R})$.
\end{proof}

To prepare for the work in Section \ref{sec:special}, we now make the reverse correspondence completely explicit for $m = 1$. Write the standard coordinates of $\mathbb{R}^2$ as $(\tilde{x},\tilde{y})$. Define $\alpha := \tilde{x} + \sqrt{-1}\tilde{y}$. Let $\Phi :\widetilde{U} \subset \mathbb{R}^2 \to \R$ be positive harmonic (we can think of this as the Gibbons-Hawking potential of a hyper-Kähler metric with $\R^2$ symmetry on $\widetilde{U} \times \R^2$). We may assume that there exists a function $h: \widetilde{U} \to \C$, holomorphic with respect to $\alpha$, such that $\mathrm{Im}(h_\alpha)=\Phi$. The map $\widetilde{T}:\widetilde{U}\to \R^2$ defined by
$$(x,y) := \widetilde{T}(\tilde{x},\tilde{y}):=(\mathrm{Im}(h(\tilde{x},\tilde{y})),\tilde{y})$$
locally provides a $C^\infty$ change of coordinates. It is then not difficult to see that for $\mathfrak{h}$ a primitive of $h$, i.e. $\mathfrak{h}_\alpha = h$, the function $u: U\to \R$ on $U:=\widetilde{T}(\widetilde{U})$ defined by
\begin{equation}\label{u}
u(x,y):=\tilde{x}\mathrm{Im}(h(\tilde{x},\tilde{y}))- \mathrm{Im}(\mathfrak{h}(\tilde{x},\tilde{y}))
\end{equation}
is a convex function satisfying $\det\hspace{0.25mm}{\rm Hess}(u) = 1$ on the domain $U$.

\section{Examples and consequences}\label{sec:special}

\subsection{Examples of two-dimensional Monge-Ampère metrics}

\begin{example}[Bryant example]\label{Exn=1}
Let $\widetilde{U}=\mathfrak{H}$ be the upper half-plane with coordinates $(\tilde{x},\tilde{y} )$. Consider the positive harmonic function $\Phi(\tilde{x},\tilde{y})=\tilde{y}$ on $\widetilde{U}$. The metric $\Phi(d\tilde{x}^2+d\tilde{y}^2)$ was studied by Bryant \cite{Bryantexample}.
With $\alpha = \tilde{x} + \sqrt{-1}\tilde{y}$ as above, the relevant holomorphic functions are
$$h(\alpha) = \frac{\alpha^2}{2}, \quad \mathfrak{h}(\alpha) = \frac{\alpha^3}{6}.$$
The map $\widetilde{T}: \widetilde{U}\to \widetilde{T} (\widetilde{U}) = U$ is given by
$$(x,y)=\widetilde{T}(\tilde{x}, \tilde{y})=(\tilde{x} \tilde{y},\tilde{y}).$$
In particular, $U=\mathfrak{H}$.
By \eqref{u} we conclude
$$
u(x,y) =  \frac{\tilde{x}^2\tilde{y}}{2}+\frac{\tilde{y}^3}{6} = \frac{x^2}{2y}+\frac{y^3}{6}.$$
The function $u$ is a convex solution of the Monge-Ampère equation on the upper half-plane $\widetilde{T}(\widetilde{U})=\mathfrak{H}$ such that the associated Monge-Ampère metric $g$ is isometric to $\Phi(d\tilde{x}^2 + d\tilde{y}^2)$ via $\widetilde{T}$. ${\rm Hess}(u)$ is the identity exactly at the point $(x,y)=(0,1)$. By Lemma \ref{SCT}, $R_g = u_{xxx}^2 + u_{xxy}^2 = 1$ at this point. Thus, because $\delta(0,1) = 1$, this example proves that Calabi's constant $M_2$ satisfies $M_2 \geq 1$.

Since scalar curvature is additive under Riemannian products, the $m$-fold product of this solution with itself yields a solution on a domain in $\mathbb{R}^{2m}$ proving that $M_{2m} \geq \sqrt{m}$ even under the assumption ${\rm Hess}(u)\in {\rm Sp}(2m,\R)$. 
In the formalism of Section \ref{2.5}, this product solution is given by
$$Z({\alpha})={\rm diag}(\alpha_1,\ldots,\alpha_m) \in \mathfrak{H}_m,  \ \ \alpha = (\alpha_1,\ldots,\alpha_m) \in \widetilde{U}=\mathfrak{H} \times \cdots \times \mathfrak{H}.$$ 
\end{example}

\begin{example}[Extremal Schwarz example]\label{ex:extremal:schwarz}
For a fixed $R_0 > 0$ consider the holomorphic function
$$g(\alpha)=1-\sqrt{1-\frac{4\alpha}{R_0}},$$ 
where the branch of the square root is chosen so that $g(0)=0$. Set
$$\widetilde{U}=\left\{ \alpha=R_0\left(t\frac{e^{i\theta}}{2}-t^2\frac{e^{2i\theta}}{4}\right): t\in[0,1), \theta\in[0,2\pi)\right\} \subset \mathbb{C}.$$
This is the domain interior to a cardioid:

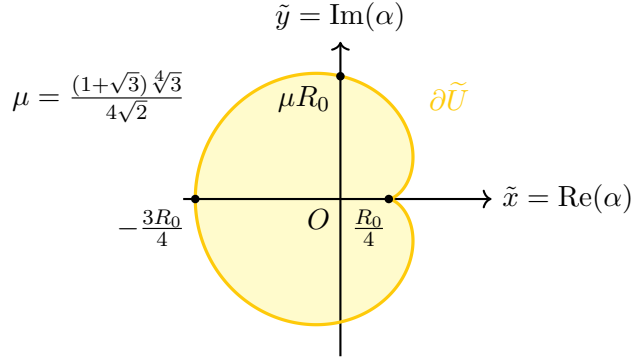
\begin{figure}[h]
\centering
\begin{tikzpicture}[scale=0.8]
\def\R{3.2}

\fill[yellow!25, domain=0:360, samples=500, smooth, variable=\t]
  plot ({\R/4*(2*cos(\t)-cos(2*\t))},
        {\R/4*(2*sin(\t)-sin(2*\t))});

\draw[->, thick] (-2.6,0) -- (2.5,0) node[right] {$\tilde{x} = {\rm Re}(\alpha)$};
\draw[->, thick] (0,-2.6) -- (0,2.6) node[above] {$\tilde{y} = {\rm Im}(\alpha)$};

\draw[very thick, yellow!60!orange, domain=0:360, samples=500, smooth, variable=\t]
  plot ({\R/4*(2*cos(\t)-cos(2*\t))},
        {\R/4*(2*sin(\t)-sin(2*\t))});

\fill (0,0) circle (0pt) node[below left] {$O$};
\fill (-2.4,0) circle (2pt) node[below left] {$-\frac{3R_0}{4}$};
\fill (0.8,0) circle (2pt) node[below left, xshift=1mm] {$\frac{R_0}{4}$};
\fill (0,2.03) circle (2pt) node[below left] {$\mu R_0$};

\node[black] at (-4,1.7) {$\mu = \frac{(1+\sqrt{3})\sqrt[4]{3}}{4\sqrt{2}}$};
\node[yellow!60!orange] at (1.8,1.7) {$\partial \widetilde{U}$};
\end{tikzpicture}
\end{figure}

\noindent It is easy to check that $g(\widetilde{U}) = \mathbb{D}$, the unit disk. Indeed,
$$\alpha=R_0\left(t\frac{e^{i\theta}}{2}-\frac{t^2e^{2i\theta}}{4}\right)  \Longrightarrow  g(\alpha) =  1-\sqrt{1-\frac{4\alpha}{R_0}}= 1-\sqrt{1-2te^{i\theta}+t^2e^{2i\theta}}= te^{i\theta}.
$$
Thus, by applying the inverse of the Cayley transformation to $g(\alpha)$, we obtain a holomorphic function defined on $\widetilde{U}$ whose image is the upper half-plane:
\begin{equation}\label{holschw}
Z(\alpha) = (\Gamma + i\Phi)(\alpha) = \mathscr{C}^{-1}(g(\alpha))= i\frac{1+g(\alpha)}{1-g(\alpha)}=-i+\frac{2i}{\sqrt{1-\frac{4\alpha}{R_0}}}.
\end{equation}
Let $h$ denote the primitive of this function given by
\begin{equation}\label{eq:schwarz_h} h(\alpha)=-i\alpha+iR_0\left(1-\sqrt{1-\frac{4\alpha}{R_0}}\right).\end{equation}
Following the construction of Section \ref{2.5}, we obtain a local diffeomorphism 
\begin{equation}\label{eqxy1}
\widetilde T(\tilde{x},\tilde{y})=\left(\operatorname{Im}(h(\tilde{x},\tilde{y})),\tilde{y}\right)=(x,y)
\end{equation}
and a convex local solution $u(x,y)$ of the Monge-Ampère equation satisfying
\begin{equation}\label{eqxy2}
   u_x=\tilde{x}, \quad u_y=-{\operatorname{Re}(h(\tilde{x},\tilde{y}))}.
\end{equation}
Since $\operatorname{Im}(h)_{\tilde{x}}=\Phi>0$ and since $\widetilde U$ is horizontally convex, the map $\widetilde{T}$ is injective and is therefore a diffeomorphism onto its image. Thus, $u$ is a convex solution of the Monge-Ampère equation defined on the entire image $U = \widetilde{T}(\widetilde{U})$.
Now introduce the Nitsche data
$$\sigma=u_x+x+i(u_y+y), \quad f(\sigma)=x-u_x-i(y-u_y).$$
By \eqref{eqxy2}, \eqref{eqxy1} and \eqref{eq:schwarz_h},
$$
\sigma  =  \alpha-ih(\alpha) = R_0\left(1 -\sqrt{1-\frac{4\alpha}{R_0}}\right), \quad
f(\sigma)  =  -\alpha-ih(\alpha) = \sigma - 2\alpha.
$$
The first equation gives $\alpha$ in terms of $\sigma$,
\begin{equation}\label{alphasig}
\alpha=\frac{\sigma}{2}-\frac{\sigma^2}{4R_0}.
\end{equation}
Substituting this into the second equation, we get
$$f(\sigma)= \frac{\sigma^2}{2R_0}, \quad f'(\sigma)=\frac{\sigma}{R_0}.$$ 
The map $\alpha \mapsto \sigma$ is a biholomorphism onto its image, defined on all of $\widetilde{U}$. By the definition of $\widetilde{U}$, every $\alpha \in \widetilde{U}$ can be written as
$$\alpha=R_0\left(t\frac{e^{i\theta}}{2}-t^2\frac{e^{2i\theta}}{4}\right), \quad t \in [0,1), \quad \theta \in [0,2\pi).$$
Then clearly $\sigma=R_0 te^{i\theta}$, so the image of $\widetilde{U}$ in the $\sigma$-plane is the disk $B_{R_0}(0) = T_3(U)$. Since $|f'(\sigma)| < 1$ by the remark after Lemma \ref{Holchart1}, the domain $\widetilde{U}$ cannot be extended, although this was already clear because we require that $\Phi >0$ and $(\mathscr{C}^{-1} \circ g)(\widetilde{U})$ is the entire upper half-plane.

We can also determine the image $U = \widetilde{T}(\widetilde{U})$ in the $(x,y)$-plane:
$$x=\operatorname{Re}(\sigma)-\operatorname{Re}(\alpha), \quad y=\operatorname{Im}(\alpha),$$
so in terms of the parameters $t,\theta$,
$$
x+iy = \frac{R_0t}{2}e^{i\theta} + \frac{R_0t^2}{4} e^{-2i\theta}.$$
For $t = 1$ this is the equation of a deltoid:

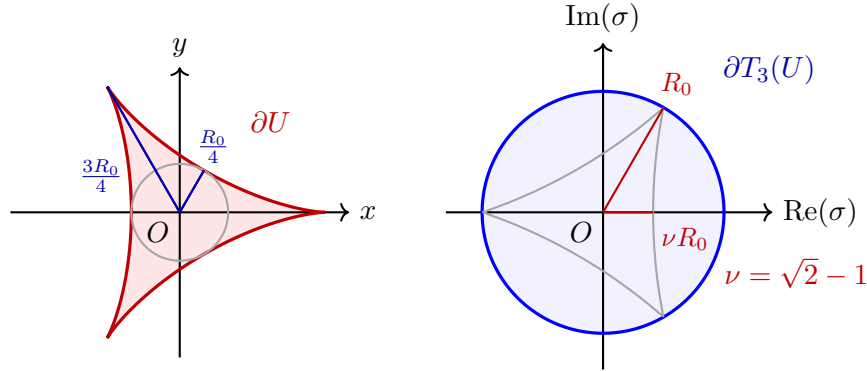
\vspace{-2mm}
\begin{figure}[h]
\centering
\begin{tikzpicture}[scale=0.8]
\begin{scope}
\def\R{3.2}

\def\tmin{60}
\def\tmax{120}

\fill[red!10, domain=0:360, samples=500, smooth, variable=\t]
  plot ({\R/4*(2*cos(\t)+cos(2*\t))},
        {\R/4*(2*sin(\t)-sin(2*\t))});

\draw[->, thick] (-2.8,0) -- (2.8,0) node[right] {$x$};
\draw[->, thick] (0,-2.4) -- (0,2.4) node[above] {$y$};

\draw[very thick, red!75!black, domain=0:360, samples=500, smooth, variable=\t]
  plot ({\R/4*(2*cos(\t)+cos(2*\t))},
        {\R/4*(2*sin(\t)-sin(2*\t))});

\draw[thick, gray!70] (0,0) circle ({\R/4});

\fill (0,0) circle (0pt) node[below left] {$O$};

\coordinate (Pmin) at ({\R/4*(2*cos(\tmin)+cos(2*\tmin))},
                       {\R/4*(2*sin(\tmin)-sin(2*\tmin))});

\coordinate (Pmax) at ({\R/4*(2*cos(\tmax)+cos(2*\tmax))},
                       {\R/4*(2*sin(\tmax)-sin(2*\tmax))});

\draw[thick, blue!70!black] (0,0) -- (Pmin);
\draw[thick, blue!70!black] (0,0) -- (Pmax);

\draw[
  thick,
  blue!70!black
]
(0,0) -- (Pmin)
node[pos=0.5, font=\small, xshift=8pt, yshift=17pt]
{$\frac{R_0}{4}$};

\draw[
  thick,
  blue!70!black
]
(0,0) -- (Pmax)
node[pos=0.5, font=\small, xshift=-16pt, yshift=-10pt]
{$\frac{3R_0}{4}$};

\node[red!75!black] at (1.5,1.5) {$\partial U$};
\end{scope}
\begin{scope}[xshift=7.0cm]
    \def\R{2}

\fill[blue!5] (0,0) circle (\R);

\draw[->, thick] (-\R-0.6,0) -- (\R+0.8,0) node[right] {${\rm Re}(\sigma)$};
\draw[->, thick] (0,-\R-0.6) -- (0,\R+0.8) node[above] {${\rm Im}(\sigma)$};

\draw[very thick, blue, domain=0:360, samples=300, smooth, variable=\t]
  plot ({\R*cos(\t)},{\R*sin(\t)});

\draw[thick, gray!70, domain=0.41422:1, samples=300, smooth, variable=\t]
  plot ({\R*\t*cos((1/3)*acos((1-4*\t^2-\t^4)/(4*\t^3)))},{\R*\t*sin((1/3)*acos((1-4*\t^2-\t^4)/(4*\t^3)))});
\draw[thick, gray!70, domain=0.41422:1, samples=300, smooth, variable=\t]
  plot ({\R*\t*cos((1/3)*(360-acos((1-4*\t^2-\t^4)/(4*\t^3))))},{\R*\t*sin((1/3)*(360-acos((1-4*\t^2-\t^4)/(4*\t^3))))});
\draw[thick, gray!70, domain=0.41422:1, samples=300, smooth, variable=\t]
  plot ({\R*\t*cos((1/3)*(360+acos((1-4*\t^2-\t^4)/(4*\t^3))))},{\R*\t*sin((1/3)*(360+acos((1-4*\t^2-\t^4)/(4*\t^3))))});
\draw[thick, gray!70, domain=0.41422:1, samples=300, smooth, variable=\t]
  plot ({\R*\t*cos((1/3)*(720-acos((1-4*\t^2-\t^4)/(4*\t^3))))},{\R*\t*sin((1/3)*(720-acos((1-4*\t^2-\t^4)/(4*\t^3))))});
\draw[thick, gray!70, domain=0.41422:1, samples=300, smooth, variable=\t]
  plot ({\R*\t*cos((1/3)*(720+acos((1-4*\t^2-\t^4)/(4*\t^3))))},{\R*\t*sin((1/3)*(720+acos((1-4*\t^2-\t^4)/(4*\t^3))))});
\draw[thick, gray!70, domain=0.41422:1, samples=300, smooth, variable=\t]
  plot ({\R*\t*cos((1/3)*(1080-acos((1-4*\t^2-\t^4)/(4*\t^3))))},{\R*\t*sin((1/3)*(1080-acos((1-4*\t^2-\t^4)/(4*\t^3))))});
  
\fill (0,0) circle (0pt) node[below left] {$O$};

\draw[thick, red!75!black] (0,0) -- ({\R*cos(60)},{\R*sin(60)})
node[pos=1, font=\small, xshift=5pt, yshift=9pt]
{$R_0$};

\draw[thick, red!75!black] (0,0) -- ({\R*0.41422},{0})
node[pos=1, font=\small, xshift=12pt, yshift=-10pt]
{$\nu R_0$};

\node[red!75!black] at (3.2,-1) {$\nu = \sqrt{2}-1$};

\draw[
  thick,
  blue!70!black
];

\node[blue!70!black] at (2.75,2.35) {$\partial T_{3}(U)$};
\end{scope}
\end{tikzpicture}
\end{figure}

Lastly, like the previous Example \ref{Exn=1}, this example also shows that $M_2 \geq 1$: we have $g(0) = 0$, hence $Z(0) = \Gamma(0) + i\Phi(0) = i$, and hence ${\rm Hess}(u)|_O = I_2$. In particular, $\delta(O)$ is the Euclidean inradius of the deltoid, which we can easily check is equal to $R_0/4$: 
$$x + iy =  \frac{R_0}{2}e^{i\theta}+\frac{R_0}{4}e^{-2i\theta} \;\,\Longrightarrow\;\,
 |x+iy| = \frac{R_0}{4}\sqrt{5+4\cos(3\theta)} \in \left[\frac{R_0}{4},\frac{3R_0}{4}\right].
$$
(With more work, one can also show that the constant $R$ in the sense of Lemma \ref{3NN}, i.e. the inradius of $T_3(B_{\delta(O)}(O))$, is given by $R = (\sqrt{2}-1)R_0$.) Moreover, from \eqref{f''}, 
$$\sqrt{R_g(O)} = 4 |\partial_\sigma \mathscr{C}(Z)(O)| = 2|\partial_\alpha g(O)| = \frac{4}{R_0} = \frac{1}{\delta(O)}.$$
Thus, again, the optimal constant $M_2$ in Theorem \ref{thm:main} cannot be smaller than $1$.
\end{example}

The relevance of Example \ref{ex:extremal:schwarz} is explained by the proof of the following result (cf. Theorem \ref{thm:main:3_not_sharp}).

\begin{theorem}\label{T2B}
    There exists no example realizing the bound $M_2 \leq 3$ of Theorem \ref{thm:main} for $m = 1$.
\end{theorem}

\begin{proof}
    Suppose by contradiction that there exists a convex solution $u:U \subset \R^2\to \mathbb{R}$ of the Monge-Ampère equation \eqref{MA}, normalized as in \eqref{symplectic}, such that 
$$\sqrt{R_g(O)}=\frac{3}{\delta(O)}.$$
Replace $U$ by $B_{\delta(O)}(O)$. Since $U$ is then convex, the map $T_{3}:U\to T_{3}(U)$ defined by \eqref{T3m} is injective, and is therefore a diffeomorphism onto its image.
Let $f(\sigma)$ be the holomorphic function introduced in Lemma \ref{Holchart1}. By the same lemma and by \eqref{f''},
$$|f''(O)| = \frac{1}{4}\sqrt{R_g(O)}=\frac{3}{4\delta(O)}.$$
Let $R>0$ be the maximal radius such that $B_{R}(O)\subset T_{3}(U)$. Since $f'(\sigma)$ is holomorphic with $f'(O)=0$ and $|f'(\sigma)| < 1$, the Schwarz lemma yields that
$$|f''(O)|\leq \frac{1}{R}.$$
Combining the previous two inequalities with Lemma \ref{3NN}, we get $\delta(O)=3R/4$ and
$|f''(O)|=1/R$. Consequently, the equality case in the Schwarz lemma is attained. Thus, without loss of generality,
$$f'(\sigma)=\frac{\sigma}{R},\quad f(\sigma)=\frac{\sigma^{2}}{2R}.$$
Moreover, $|f'(\sigma)|<1$ implies $|\sigma|<R$, i.e. $T_{3}(U)\subset B_{R}(O)$; but since $B_{R}(O)\subset T_{3}(U)$ by definition, we conclude that $T_{3}(U)=B_{R}(O)$. This means we are in the setting of Example \ref{ex:extremal:schwarz} with $R_0 = R$, and hence $\delta(O) = R_0/4 = R/4$ by the discussion there (and $U$ is not a disk). This is a contradiction.
\end{proof}

\subsection{Curvature estimates in two-dimensional Monge-Ampère metrics} 

In the remainder of this paper we are going to exploit the fact that the partial Legendre transform also yields a convenient representation of the Monge-Ampère metric tensor: by Corollary \ref{MAP}, any Monge-Ampère metric $g$ on a domain in $\mathbb{R}^2$ is isometric to a metric of the form
$$g = \Phi(\tilde{x},\tilde{y})(d\tilde{x}^2+d\tilde{y}^2),$$
where $\Phi$ is a positive harmonic function, and by Proposition \ref{MACorr} the converse holds as well.

\begin{lemma} \label{GC}
We can express the scalar curvature of the Monge-Ampère metric $g$ as
\begin{align*}
    R_g=\frac{g(d\Phi,d \Phi)}{\Phi^2}\geq0.
\end{align*}
\end{lemma}

\begin{proof}
Since $\Phi$ is positive, we can write $\Phi=e^{H}$. Then, by standard general formulas,
\begin{equation}\label{LBO+CHh}
    \Delta_g=e^{-H}(\partial_{\tilde{x}}^2+\partial_{\tilde{y}}^2), \quad R_g=-\Delta_g H.
\end{equation}
Since $\Phi$ is harmonic in the coordinates $\tilde{x},\tilde{y}$, it is harmonic with respect to $g$ as well. Thus,
\begin{align*}\begin{split}
0 &= \Delta_g e^{H} =  e^{-H}((e^{H})_{\tilde{x}\tilde{x}}+ (e^{H})_{\tilde{y}\tilde{y}})= e^{-H}(e^{H}(H_{\tilde{x}}^2+H_{\tilde{y}}^2) + e^{H}(H_{\tilde{x}\tilde{x}}+H_{\tilde{y}\tilde{y}})) \\
& = e^{H}g(dH, dH) +e^{H}\Delta_g H = e^{H}g(dH, dH) -e^{H}R_g.
\end{split}\end{align*}
This implies the claim because $dH = \Phi^{-1}d\Phi$.
\end{proof}

In the rest of this paper we will simply write $\Delta, R$ instead of $\Delta_g, R_g$. Then the following theorem is our main result in this section (cf. Theorem \ref{main:curv_diffeq}).

\begin{theorem}\label{alpha>0}
On the open set $\{R>0\}$ we have that
$$\Delta R^{\alpha}= (\alpha^2N+3\alpha)R^{\alpha+1}$$
for every $\alpha > 0$, where
\begin{align*}\begin{split}N&:=4h^{-2}\Phi^2\biggl(\biggl(\Phi_{\tilde{x}\tilde{x}}-\frac{3}{2}\frac{\Phi_{\tilde{x}}^2-\Phi_{\tilde{y}}^2}{\Phi}\biggr)^2+\biggl(\Phi_{\tilde{x}\tilde{y}}-3\frac{\Phi_{\tilde{x}}\Phi_{\tilde{y}}}{\Phi}\biggr)^2\biggr)\geq0,\\
h&:= \Phi_{\tilde{x}}^2+\Phi_{\tilde{y}}^2 = \Phi^3 R > 0.\end{split}\end{align*}
In particular, on the set $\{R > 0\}$,
$$\Delta R^{\alpha}\geq 3\alpha R^{\alpha+1}.$$
Moreover, this inequality holds globally in the lower barrier sense.
\end{theorem}

\begin{proof}
We work on the open set where $R>0$. By Lemma \ref{GC}, 
$$R=\frac{g(d\Phi,d\Phi)}{\Phi^2}=\frac{\Phi_{\tilde{x}}^2+\Phi_{\tilde{y}}^2}{\Phi^3}=\frac{h}{\Phi^3}>0.$$
Using \eqref{LBO+CHh}, the computation of $\Delta R^{\alpha}$ reduces to computing the Euclidean Laplacian of $R^\alpha=h^{\alpha}\Phi^{-3\alpha}$. We compute the derivatives in the variable $\tilde{x}$; the derivatives with respect to $\tilde{y}$ are analogous. Thus,
$$\begin{array}{ccl}
 (R^{\alpha})_{\tilde{x}}  =\alpha h^{\alpha-1}h_{\tilde{x}}\Phi^{-3\alpha}-3\alpha h^{\alpha}\Phi^{-3\alpha-1}\Phi_{\tilde{x}}.
\end{array}$$
Differentiating again, we get
\begin{align*}\begin{split}
 (R^{\alpha})_{\tilde{x}\tilde{x}} &= \alpha(\alpha-1)h^{\alpha-2}h_{\tilde{x}}^2\Phi^{-3\alpha}+\alpha h^{\alpha-1}h_{\tilde{x}\tilde{x}}\Phi^{-3\alpha}\\
     &-6\alpha^2 h^{\alpha-1}h_{\tilde{x}}\Phi^{-3\alpha-1}\Phi_{\tilde{x}} +3\alpha(3\alpha+1)h^{\alpha}\Phi^{-3\alpha-2}\Phi_{\tilde{x}}^2-3\alpha h^{\alpha}\Phi^{-3\alpha-1}\Phi_{\tilde{x}\tilde{x}}.    
\end{split}\end{align*}
Combining the $\tilde{x}$- and $\tilde{y}$-derivatives and using that $\Phi$ is harmonic with respect to $(\tilde{x},\tilde{y})$, we obtain
\begin{align*}\begin{split}
  \Phi\Delta R^{\alpha}  & =  \displaystyle (R^{\alpha})_{\tilde{x}\tilde{x}} + (R^{\alpha})_{\tilde{y}\tilde{y}} \\
    & =  \displaystyle \alpha(\alpha-1)h^{\alpha-2}(h_{\tilde{x}}^2+h_{\tilde{y}}^2)\Phi^{-3\alpha}+\alpha h^{\alpha-1}(h_{\tilde{x}\tilde{x}}+h_{\tilde{y}\tilde{y}})\Phi^{-3\alpha} \\
      &  \displaystyle -6\alpha^2 h^{\alpha-1}\Phi^{-3\alpha-1}(h_{\tilde{x}}\Phi_{\tilde{x}}+h_{\tilde{y}}\Phi_{\tilde{y}})+3\alpha(3\alpha+1)h^{\alpha+1}\Phi^{-3\alpha-2}.
\end{split}\end{align*}
Using the harmonicity of $\Phi$, we can calculate that
\begin{align*}\begin{split}
   h_{\tilde{x}}^2+h_{\tilde{y}}^2  & =   4h(\Phi_{\tilde{x}\tilde{x}}^2+\Phi_{\tilde{x}\tilde{y}}^2),\\
    h_{\tilde{x}\tilde{x}}+h_{\tilde{y}\tilde{y}} & =  4(\Phi_{\tilde{x}\tilde{x}}^2+\Phi_{\tilde{x}\tilde{y}}^2), \\
   h_{\tilde{x}}\Phi_{\tilde{x}}+h_{\tilde{y}}\Phi_{\tilde{y}}  & = 2\Phi_{\tilde{x}\tilde{x}}(\Phi_{\tilde{x}}^2-\Phi_{\tilde{y}}^2)+4\Phi_{\tilde{x}}\Phi_{\tilde{y}}\Phi_{\tilde{x}\tilde{y}}.
\end{split}\end{align*}
Thus, by substitution,
\begin{align*}\begin{split}
  \Phi\Delta R^{\alpha} 
      & =  \alpha^2[4h^{-2}\Phi^2(\Phi_{\tilde{x}\tilde{x}}^2+\Phi_{\tilde{x}\tilde{y}}^2)-12h^{-2}\Phi(\Phi_{\tilde{x}\tilde{x}}(\Phi_{\tilde{x}}^2-\Phi_{\tilde{y}}^2)+2\Phi_{\tilde{x}}\Phi_{\tilde{y}}\Phi_{\tilde{x}\tilde{y}})+9]h^{\alpha+1}\Phi^{-3\alpha-2}\\
      &\displaystyle +3\alpha h^{\alpha+1}\Phi^{-3\alpha-2}.
\end{split}\end{align*}
A direct expansion shows that $N$, defined as a sum of squares as above, is actually equal to the term in square brackets, so the claimed formula for $\Delta R^\alpha$ follows. That $\Delta R^\alpha \geq 3\alpha R^{\alpha + 1}$ continues to hold in the barrier sense at any point $p$ with $R(p) = 0$ is trivial because $R \geq 0$ everywhere by Lemma \ref{GC}, so the constant function zero serves as a lower barrier for $R^\alpha$ at $p$.
\end{proof}

\begin{example}\label{n=2} 
Consider the harmonic function $\Phi(\tilde{x},\tilde{y})=1+2\tilde{x}\tilde{y}$ restricted to the domain $\{\Phi > 0\}$. Since $p = (1/2, 1/2)$ is not a critical point of $\Phi$, we have that $R(p)>0$ by Lemma \ref{GC}. Moreover, it is easy to see that $N(p)=0$. Thus, $\Delta R^\alpha = 3\alpha R^{\alpha+1}$ can hold even at a point where $R > 0$.
\end{example}

\subsection{Comparison with some results of Calabi}

Our Theorem \ref{alpha>0} improves a result due to Calabi \cite{EC}, which was proved using a different method, i.e. without using the conformal representation of $g$. In this section we will compare the two results and derive some consequences.

\begin{lemma}[Calabi {\cite[Eq. 2.17--2.18]{EC}}]\label{alp}
    Let $g$ be the Hessian metric associated with a smooth convex solution $u$ of the real Monge-Ampère equation
\eqref{MA} on a domain in $\mathbb{R}^n$. Let $\Delta, R$ denote the Laplace-Beltrami operator and the scalar curvature of $g$, respectively. Then we have that
    $$\displaystyle \frac{1}{2}\Delta R\geq \frac{n+1}{n(n-1)}R^2+A_{ijk,\ell}A^{ijk,\ell}, \quad g(d R, d R)=4A^{abc,i}A_{abc}A_{jk\ell,i}A^{jk\ell},$$
    where $A_{ijk}$ is the tensor defined by
    $$A_{ijk}:=-\frac{1}{2} \frac{\partial^3 u}{\partial x_{i}\partial x_{j}\partial x_{k}},$$
    and it is such that
    $$g^{ij}A_{ijk}=0, \quad A_{ijk,\ell}=A_{ij\ell,k}.$$
\end{lemma}

\begin{proposition}[Calabi {\cite[Eq. 2.15]{EC}} for $\alpha=1/2$]\label{prop:DiffIneqR}
For all $\alpha \geq 1/2$ the inequality
\begin{equation}\label{LE}
\Delta R^{\alpha} \geq 2\alpha\frac{n+1}{n(n-1)} R^{\alpha+1}
\end{equation}
holds in the lower barrier sense.
\end{proposition}

\begin{proof}
As in the proof of Theorem \ref{alpha>0}, it is enough to prove this in the classical sense on the set where $R > 0$. It is straightforward to compute that
$$\Delta R^\alpha = \alpha(\alpha-1)R^{\alpha-2}g(d R,d R)+\alpha R^{\alpha-1}\Delta R.$$
Thus, by using Lemma \ref{alp},
$$\Delta R^\alpha\geq 2\alpha\frac{n+1}{n(n-1)}R^{\alpha+1}+\alpha R^{\alpha-2}(2R  A_{ijk,\ell}A^{ijk,\ell}+(\alpha-1)g(d R,d R)).$$
There are two cases. If $\alpha\geq1$, the statement follows because obviously the remaining term
\begin{equation}\label{B}
2R  A_{ijk,\ell}A^{ijk,\ell}+(\alpha-1)g(d R,d R),
\end{equation}
is non-negative. If $\alpha<1$, we would like to show that the only admissible values of $\alpha$ for which Calabi's method applies are those $\geq 1/2$ (Calabi \cite{EC} only wrote down the case $\alpha = 1/2$). To this end, we would like to write the remaining term \eqref{B} as a squared norm. We introduce the tensor
$$B_{abcijk\ell}:=CA_{abc,i}A_{jk\ell}+DA_{abc}A_{jk\ell,i},$$
where $C,D\in\R$ are constants to be determined. By Lemma \ref{alp}, we have that
\begin{align*}\begin{split}
     B_{abcijk\ell}B^{abcijk\ell} & = \displaystyle (C^2+D^2)RA_{ijk,\ell}A^{ijk,\ell}+2CDA^{abc,i}A_{abc}A_{jk\ell,i}A^{jk\ell} \\
     & = \displaystyle  (C^2+D^2)RA_{ijk,\ell}A^{ijk,\ell}+\frac{CD}{2}g(d R,d R).
\end{split}\end{align*}
Thus, in order to obtain
$$
    B_{abcijk\ell}B^{abcijk\ell}=2R  A_{ijk,\ell}A^{ijk,\ell}+(\alpha-1)g(d R,d R),
$$
we need that
$$C^2+D^2=2,\quad  \frac{CD}{2}=\alpha-1.$$ 
Equivalently,
$$C+D=\pm\sqrt{4\alpha-2},\quad  C-D=\pm\sqrt{6-4\alpha}.$$
Such real constants exist if and only if $\alpha\in[1/2,3/2]$.
\end{proof}

\begin{remark}\label{Sharp1}
Using Schur's lemma, Calabi \cite[p.~114]{EC} proved that for $n \geq 3$, equality in \eqref{LE} cannot hold on a nonempty open set where $R > 0$. Here is an alternative argument for all $n \geq 2$ without using Schur's lemma: Let $\alpha=1/2$ for simplicity. If the claim is false, then after shrinking the open set there exist fixed indices $a,b,c$ such that $A_{abc}$ never vanishes. (This follows from $A \not \equiv 0$, which is true because otherwise $dR \equiv 0$ by Lemma \ref{alp}, so $\Delta R \equiv 0$, so $R \equiv 0$.) Moreover, from $B \equiv 0$ and $\alpha = 1/2$,
$$A_{ijk,\ell}=\frac{A_{abc,\ell}}{A_{abc}}A_{ijk}.$$
Setting $\lambda_{\ell}:=A_{abc,\ell}/A_{abc}$, we obtain $A_{ijk,\ell}=\lambda_{\ell}A_{ijk}$. Combining this with the properties of the tensor $A_{ijk}$ given in Lemma \ref{alp} and with the full symmetry of $A_{ijk}$, we further obtain for all $s,t,k$ that
$$0
 = (\lambda_{s}\lambda_{t})(g^{ij}A_{ijk})
 = g^{ij} (\lambda_{s}\lambda_{t} A_{ijk})
 = (g^{ij} \lambda_{i}\lambda_{j}) A_{stk}.
$$
Since $A_{abc}$ never vanishes, we get $\lambda_\ell \equiv 0$ for all $\ell$, so $A_{ijk,\ell} \equiv 0$ for all $i,j,k,\ell$, so $R \equiv 0$ as above.
\end{remark}

\begin{remark}
    In \cite[p.~114]{EC}, Calabi also claimed that equality in \eqref{LE} cannot hold even at a single point where $R > 0$, but we have not been able to reproduce such an argument. In fact, Example \ref{n=2} is an explicit counterexample to this stronger claim when $n=2$. The case $n \geq 3$ remains open.
\end{remark}

Calabi used Proposition \ref{prop:DiffIneqR} with $\alpha = \beta/2$ and a Riccati comparison argument to prove

\begin{theorem}[Calabi {\cite[Eq. 1.19 and bottom of p. 120]{EC}} for $\beta=1$, {\cite[Thm 2.2.13]{JSAR}} for $\beta>1$]\label{T11}
    Let $g$ be the Hessian metric associated with a smooth convex solution $u$ of the real Monge-Ampère equation
\eqref{MA} on a domain $U \subset \mathbb{R}^n$. Let $R \geq 0$ denote the scalar curvature of $g$. Then, for all $x \in U$,
$$\sqrt{R(x)}\leq \frac{1}{{\rm dist}_g(x,\partial U)}\inf_{\beta \geq 1}  s_{\max}(n,\beta)\sqrt{\frac{n(n-1)}{\beta(n+1)}},$$
where, for all $\beta > 0$, $s_{\max}(n,\beta)$ is the maximal existence time of the initial value problem
 \begin{equation}\label{ODE}
y''(s)+\frac{n-1}{s}y'(s)=y(s)^{\frac{\beta+2}{\beta}}, \quad y(0)=1,\quad  y'(0)=0.
\end{equation}
\end{theorem}

Furthermore, by using barrier functions, Calabi \cite[Eq. 3.13]{EC} was able to prove that
\begin{equation}\label{eq:smax:est}
s_{\max}(n,1) \in \begin{cases}
    [\sqrt{n},\sqrt{2n}] &   (n\geq 4),  \\
    [\sqrt{2n}, 1.3111 \cdot \sqrt{2n}]&  (n=2,3).\end{cases}
    \end{equation}
For $n=2$ we have the following estimate of $s_{\max}(2,\beta)$ for all $\beta > 0$, which is sharp in the limit $\beta \to 0$.
This was found using ChatGPT based on a weaker estimate that we had proved by hand.

\begin{lemma}\label{2beta}
For all $\beta > 0$ we have that 
$$s_{\max}(2,\beta) \in [2\sqrt{\beta}, 2\sqrt{\beta(\beta+1)}].$$
\end{lemma}

\begin{proof}
     Consider the initial value problem \eqref{ODE}, and make the change of variables
$$s=\sqrt{\beta}r,\quad  y(s)=Y(r).$$
Then $Y(r)$ solves the initial value problem
\begin{equation}\label{smax}
Y''(r)+\frac{1}{r}Y'(r)=\beta Y(r)^{\frac{\beta+2}{\beta}},\quad 
Y(0)=1,\quad  Y'(0)=0.
\end{equation}
Consider the functions
$$U_{1}(r)=\left(1-\frac{r^2}{4}\right)^{-\beta},\quad  U_{2}(r)=\left(1-\frac{r^2}{4(1+\beta)}\right)^{-\beta}.$$
The function $U_{1}(r)$ blows up at $r_{1}=2$ and the function $U_{2}(r)$ blows up at $r_{2}=2\sqrt{\beta+1}$. By explicit computations, we have that
$$U_{1}''(r)+\frac{1}{r}U_{1}'(r) = \left(1+\beta \frac{r^2}{4}\right) \cdot \beta U_{1}(r)^{\frac{\beta+2}{\beta}},\quad  U_{2}''(r)+\frac{1}{r}U_{2}'(r)= \frac{1}{\beta+1}\cdot \beta U_{2}(r)^{\frac{\beta+2}{\beta}}.$$
Thus, $U_{1}$ is a supersolution and $U_{2}$ is a subsolution of \eqref{smax}. Thus, if $r_{\max}$ denotes the blow-up time of $Y(r)$, then $r_{\max} \in [2,2\sqrt{\beta+1}]$. This yields the statement of the lemma.
\end{proof}

\begin{corollary}\label{cor:cal:improved:c}
Let $R$ be the scalar curvature of the Hessian metric $g$ associated with a smooth convex solution $u$ of the real Monge-Ampère equation
\eqref{MA} on a domain $U \subset \mathbb{R}^2$. Then, for all $x \in U$,
$$\sqrt{R(x)}\leq \frac{2\sqrt{2/3}}{{\rm dist}_g(x,\partial U)}.$$
\end{corollary}

\begin{proof}
Recall the conclusion of Theorem \ref{T11} in dimension $n=2$:
$$\sqrt{R(x)}\leq \frac{\sqrt{2/3}}{{\rm dist}_g(x,\partial U)} \inf_{\beta \geq 1} \frac{s_{\max}(2,\beta)}{\sqrt{\beta}},$$
where $s_{\max}(2,\beta)$ is defined for all $\beta > 0$ and, by Lemma \ref{2beta}, satisfies $s_{\max}(2,\beta) \sim 2\sqrt{\beta}$ as $\beta \to 0$. The reason for taking the inf only over $\beta \geq 1$ is that Proposition \ref{prop:DiffIneqR} holds only for $\alpha = \beta/2 \geq 1/2$. As Theorem \ref{alpha>0} removes this restriction, the proof of Theorem \ref{T11} gives the desired estimate.
\end{proof}

\subsection{Improved curvature estimates near the endpoint of a minimal geodesic}

    To end this paper, we now return to Calabi's best result towards the general version of Theorem \ref{thm:main}: 
    $$\sqrt{R(x)} \leq \frac{M_n}{{\rm dist}_{g|_x}(x,\partial U)}, \;\,M_n < \infty \;(n = 2,3,4,5),$$
    for all $x \in U$, where $u: U \subset \mathbb{R}^n \to \mathbb{R}$ is any smooth convex solution to $\det\hspace{0.25mm}{\rm Hess}(u)=1$ on an open set, $g$ is the associated Hessian metric and $R$ is its scalar curvature. This is deduced from Theorem \ref{T11}, but in fact the more difficult step, which only works for $n \leq 5$, is to bound ${\rm dist}_g(x,\partial U)$ below in terms of ${\rm dist}_{g|_x}(x,\partial U)$. In particular, if $U = \mathbb{R}^n$, one needs to prove that $g$ is complete.
    
    Both steps can be encoded as an estimate of the form
    \begin{align}\label{eq:crucial:curv:est}\sqrt{{\rm Ric}(C'(s),C'(s))} \leq \frac{c}{\gamma - s},\end{align}
    where $C(s)$, $s \in [0,\gamma)$, is a minimal unit-speed geodesic in the Hessian metric $g$ realizing the shortest distance $\gamma$ from $x$ to $\partial U$ in this metric. Theorem \ref{T11} and \eqref{eq:smax:est} imply \eqref{eq:crucial:curv:est} with
    \begin{equation}\label{eq:c:best} c = 1.52 \;\,(n = 2), \quad c = 3.94 \;\, (n = 3), \quad c = n\sqrt{2\frac{n-1}{n+1}} \;\,(n \geq 4).\end{equation}
 Calabi pointed out that the \emph{second} step, i.e. the lower bound of $\gamma$ in terms of the affine distance of $x$ to $\partial U$, can be vastly simplified and made to work for any $n$ if \eqref{eq:c:best} can be improved to $c < (\frac{n}{n-1})^{1/2}$; see \cite[p. 122]{EC}. Such an improvement seems to not have been known for any $n$. Corollary \ref{cor:cal:improved:c} implies this for $n = 2$ with $c = 2/\sqrt{3} \leq 1.16$. This is the statement we have recorded in Theorem \ref{thm:improved:calabi}.

As it stands, Calabi's lower bound of $\gamma$ in terms of ${\rm dist}_{g|_x}(x,\partial U)$ requires an improvement of \eqref{eq:c:best} asymptotically as $s \to \gamma$. In his proof this is provided by a very difficult lemma due to Levinson, which morally says that $c = (\frac{n-1}{4})^{1/2}$ is possible near the endpoint of any minimal geodesic in any manifold with ${\rm Ric}\geq 0$. Precisely, Levinson's lemma, which we have now eliminated for $n = 2$, is

\begin{lemma}[Levinson, in {\cite[p. 123]{EC}}] \label{LL}
Let $H$ be a continuously differentiable function of $s$ $(0<s<\gamma<\infty)$ such that $H'\geq H^2$. Then, for all $0 < \varepsilon < \gamma$ and for all $c'<2$,
$$\int_{\varepsilon}^{\gamma}\exp\left(c'\int_{\varepsilon}^{s}\sqrt{H'(s_{1})-H(s_{1})^{2}}\,ds_{1}\right)ds<\infty.$$
\end{lemma}

To clarify the use of this lemma for Calabi's problem, let us prove a variant (cf. Proposition \ref{prop:levin:mod}):

\begin{proposition}\label{2.3.10}
For any Riemannian manifold $M^n$ and minimal unit-speed geodesic $C: [0,\gamma) \to M$,
$$\liminf_{s \to \gamma} \;(\gamma-s)^2{\rm Ric}(C'(s),C'(s)) \leq \frac{n-1}{4}.$$
Equality can hold on a surface of positive curvature.
\end{proposition}

\begin{proof}
Suppose $C: [0,\gamma) \to M$ is a counterexample to the inequality. Thus, there exists a $\gamma_0 \in [0,\gamma)$ and an $a \in (\frac{n-1}{4},\infty)$ such that ${\rm Ric}(C'(s),C'(s)) \geq a(\gamma -s)^{-2}$ for all $s \in [\gamma_0,\gamma)$. For $s \in (0,\gamma)$ let $H(s)$ denote the mean curvature at $C(s)$ of the geodesic sphere of radius $s$ centered at $C(0)$, normalized so that $H(s) = -1/s$ in flat $\mathbb{R}^n$. By the Riccati inequality from Riemannian geometry,
\begin{align*}
H'(s)-H(s)^2\geq
\frac{1}{n-1}\mathrm{Ric}(C'(s),C'(s)) \geq \frac{a}{n-1}\frac{1}{(\gamma-s)^2}
\end{align*}
for all $s \in [\gamma_0,\gamma)$. Also, by Lemma \ref{LL}, for all $\varepsilon \in (\gamma_0, \gamma)$ and for all $c'<2$,
$$\int_{\varepsilon}^{\gamma}\exp\left(c'\int_{\varepsilon}^{s}\sqrt{H'(s_{1})-H(s_{1})^{2}}\,ds_{1}\right)ds<\infty.$$
By combining these two inequalities and evaluating the inner integral, we deduce that
$$\infty > \int_\varepsilon^\gamma \exp \left(c'\sqrt{\frac{a}{n-1}} \log \frac{\gamma-\varepsilon}{\gamma - s}\right) ds = C \int_\varepsilon^\gamma (\gamma - s)^{-c'\sqrt{\frac{a}{n-1}}} \,ds$$
for some constant $C>0$. As $a > (n-1)/4$, we can choose $c'<2$ such that $c'\sqrt{a/(n-1)}$ is still greater than $1$, making the integral divergent. This is a contradiction.

To discuss the equality case, consider the rotationally symmetric Riemannian metric 
\begin{align*}\begin{split}g(s,\theta):=ds^2+h(s)^2d\theta^2, \quad h(s):=-\sqrt{1-s}\log(1-s), \quad s \in [0,1), \quad \theta \in \mathbb{R}/2\pi\mathbb{Z}.
\end{split}\end{align*}
Since $h(0)=0$, $h'(0)=1$ and $h''(0)=0$, the metric $g$ extends as a $C^3$ Riemannian metric to the origin. Since $\lim_{s\to 1} h(s)=0$, all meridians converge to the same point in the metric completion as $s \to 1$. The Gaussian curvature $K(s)$ along each meridian is given by
$$K(s)=-\frac{h''(s)}{h(s)} = \frac{1}{4(1-s)^2}.$$
Since ${\rm Ric} = Kg$ in dimension $2$, the claim follows.\end{proof}

We believe that Levinson's lemma is far from sharp for Monge-Ampère metrics.

\begin{example}\label{ex:ric:ma}
We revisit the Bryant Example \ref{Exn=1}. According to Bryant \cite{Bryantexample}, a geodesic in the metric $\Phi(d\tilde{x}^2+d\tilde{y}^2)$ is either a vertical line or a parabola that never touches the boundary:

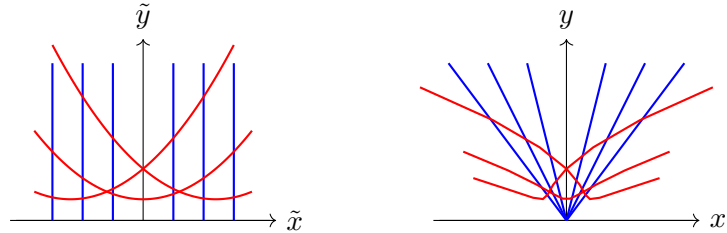
\begin{center}
\begin{tikzpicture}[scale=0.8]
\begin{scope}[xshift=2.0cm]
    
    \draw[->] (-2.2,0) -- (2.2,0) node[right] {$\tilde{x}$};
    \draw[->] (0,0) -- (0,3.0) node[above] {$\tilde{y}$};

    \draw (-2.1,0) -- (2.1,0);

    \foreach \a in {-1.5,-1.0,-0.5,0.5,1.0,1.5} {
        \draw[blue, thick] (\a,2.6) -- (\a,0);
    }

    \foreach \a in {-1.2} {
        \draw[red, thick, domain=-1.8:1.5, samples=70]
        plot ({\x},{0.35+0.35*(\x-\a)^2});
    }

    \foreach \a in {0} {
        \draw[red, thick, domain=-1.8:1.8, samples=70]
        plot ({\x},{0.35+0.35*(\x-\a)^2});
    }

    \foreach \a in {1.2} {
        \draw[red, thick, domain=-1.5:1.8, samples=70]
        plot ({\x},{0.35+0.35*(\x-\a)^2});
    }
\end{scope}

\begin{scope}[xshift=9.0cm]
    
    \draw[->] (-2.2,0) -- (2.2,0) node[right] {$x$};
    \draw[->] (0,0) -- (0,3.0) node[above] {$y$};

    \draw (-2.1,0) -- (2.1,0);

    \foreach \a in {-0.75,-0.5,-0.25,0.25,0.5,0.75} {
        \draw[blue, thick] (0,0) -- ({\a*2.6},2.6);
    }

    \foreach \a in {0} {
        \draw[red, thick, domain=-1.5:1.5, samples=10]
        plot ({\x*(0.35+0.35*(\x-\a)^2)},{0.35+0.35*(\x-\a)^2});
    }

     \foreach \a in {-1.2} {
        \draw[red, thick, domain=-2.2:1.1, samples=10]
        plot ({\x*(0.35+0.35*(\x-\a)^2)},{0.35+0.35*(\x-\a)^2});
    }

 \foreach \a in {1.2} {
        \draw[red, thick, domain=-1.1:2.2, samples=10]
        plot ({\x*(0.35+0.35*(\x-\a)^2)},{0.35+0.35*(\x-\a)^2});
    }

    \pgfmathsetmacro{\xtilde}{1.0}
    \pgfmathsetmacro{\ytilde}{2.2}
    \pgfmathsetmacro{\xpoint}{\xtilde*\ytilde}
\end{scope}
\end{tikzpicture}
\end{center}

\noindent Moreover, the Gaussian curvature is given by
$$K(\tilde{x},\tilde{y})=\frac{1}{2\tilde{y}^{3}}.$$
For our purposes, we only need the vertical unit-speed geodesic from the point $(0,1)$ to the boundary $\tilde{y}=0$, which is automatically minimal because of the translational symmetry in $\tilde{x}$:
    $$C(s)=(\tilde{x}(s),\tilde{y}(s)) = \biggl(0,\biggl(1-\frac{3}{2}s\biggr)^{2/3}\biggr),\quad  s\in (0,\gamma), \quad \gamma = \frac{2}{3}.$$
Therefore, the Gaussian curvature along $C(s)$ is given by
$$K(s)=\displaystyle \frac{2}{9}\frac{1}{\left(\gamma-s\right)^2}.$$
\end{example}

\begin{example}\label{ex:schwarz:2}
We now revisit the extremal Schwarz Example \ref{ex:extremal:schwarz}. From \eqref{alphasig} and \eqref{holschw},
$$\Phi(\sigma)=\frac{R_0^2-|\sigma|^2}{|R_0-\sigma|^2}, \quad d\alpha=\frac{R_0-\sigma}{2R_0}d\sigma.$$
Thus, the metric $g = \Phi|d\alpha|^2$ is written in terms of the $\sigma$-coordinate as
$$g=\frac{R_0^2-|\sigma|^2}{4R_0^2}|d\sigma|^2.$$
Since $g$ is still conformal, its scalar curvature $R$ is given by
$$R=-\Delta \log \left(\frac{R_0^2-|\sigma|^2}{4R_0^2}\right)=\frac{16R_0^4}{(R_0^2-|\sigma|^2)^3}.$$
Now write $\sigma=re^{i\theta}$ with $r\in[0,R_0)$, $\theta\in[0,2\pi)$. Thus, the metric $g$ is given by 
$$g=\frac{R_0^2-r^2}{4R_0^2}(dr^2+r^2d\theta^2).$$
Since the conformal factor is independent of $\theta$, the minimal unit-speed geodesics from $O$ to the boundary are the reparametrizations by arclength of the radial segments
$$C(t)=(t\cos(\theta_{0}),t\sin(\theta_{0})), \quad t\in[0,R_0).$$
The arclength $s(t)$ of $C|_{[0,t]}$ with respect to the metric $g$ is given by
$$\begin{array}{ccl}
   s(t)   =  \displaystyle \int_{0}^{t} \sqrt{g(C'(\rho),C'(\rho))} \ d\rho =  \frac{R_0}{4}\left(\arcsin\left(\displaystyle \frac{t}{R_0}\right)+\frac{t}{R_0}\sqrt{ \displaystyle 1-\frac{t^2}{R_0^2}}\right). 
\end{array}$$
As $t\to R_0$ we have that $s(t) \to \frac{\pi R_0}{8} = \gamma$. Thus,
$$\sqrt{R(O)}=\frac{4}{R_0}=\frac{\pi/2}{\gamma}.$$
This gives the lower bound $c \geq \pi/\sqrt{8} \geq 1.11$ for the optimal constant $c$ in \eqref{eq:crucial:curv:est} in dimension $n = 2$, very close to our upper bound $c \leq 1.16$. On the other hand, one can also check that as $s \to \gamma$, \eqref{eq:crucial:curv:est} holds with a smaller and smaller constant $c$ approaching $\sqrt{2}/3 \leq 0.48$. This is consistent not only with Proposition \ref{2.3.10} (since $0.48 < 0.5$) but also with the constant $c = \sqrt{2}/3$ of the Bryant Example \ref{ex:ric:ma}. Like in that example, $\Phi$ vanishes to first order at every smooth point of $\partial \widetilde{U}$.
\end{example}

\begin{example}\label{ex:2nd:order:harmonic}
    We now consider a positive harmonic function vanishing to second order at some point of $\partial\widetilde{U}$. Specifically, let $g =\Phi(d\tilde{x}^2+d\tilde{y}^2)$, $\Phi(\tilde{x},\tilde{y}):=\tilde{x}\tilde{y}$, on the first quadrant in the $(\tilde{x},\tilde{y})$-plane. Then, 
    by Lemma \ref{GC}, the scalar curvature of $g$ is given by
    $$
      R = \frac{g(d\Phi,d\Phi)}{\Phi^2} =\displaystyle \frac{\tilde{x}^2+\tilde{y}^2}{\tilde{x}^3 \tilde{y}^3}.$$

    \noindent Fix the point $(\tilde{x}_0,\tilde{y}_0) = (1,1)$. It is easy to check that the curve
    $$C(s)= (\tilde{x}(s),\tilde{y}(s)) = ((1-s\sqrt{2})^{1/2},(1-s\sqrt{2})^{1/2}), \quad s\in\left[0,L\right), \quad L=\frac{1}{\sqrt{2}},$$
    is parametrized by arclength with respect to $g$. It is also a $g$-geodesic because it parametrizes the fixed point set of the $g$-isometry $(\tilde{x},\tilde{y}) \mapsto (\tilde{y},\tilde{x})$. The Gaussian curvature along $C(s)$ satisfies
    $$K(s) =\frac{1}{2}R(C(s))=\frac{1}{\displaystyle (1-s\sqrt{2})^{2}}= \frac{1/2}{(L-s)^2}.$$
    At first sight, this may seem to contradict Proposition \ref{2.3.10}, but $C(s)$ is actually not minimal on the interval $s \in [(1-\varepsilon^2)/\sqrt{2}, 1/\sqrt{2})$ for any given $\varepsilon \in (0,1]$. Indeed, we claim that for any $\varepsilon \in (0,1]$ there is a $\delta \in (0,\varepsilon]$ such that the length $L_1$ of the segment 
    $(\delta,\delta ) \to (\varepsilon ,\varepsilon)$
    is bigger than the length $L_2$ of the broken path
    $(\delta ,\delta ) \to (\varepsilon ,\delta ) \to (\varepsilon ,\varepsilon ).$
    To see this, we compute $$L_1=\frac{1}{\sqrt{2}}(\varepsilon^2-\delta^2), \quad L_2 =\frac{2}{3}(\varepsilon^{1/2} + \delta^{1/2})(\varepsilon^{3/2}-\delta^{3/2}).$$ Thus, for any fixed $\varepsilon \in (0,1]$ we get $L_1 > L_2$ by sending $\delta \to 0$.
    
     One can check directly that there are infinitely many conjugate points along $C(s)$. Unfortunately we were unable to compute the distance function to the boundary. The above results support our intuition that the shortest geodesic to the boundary favors Bryant behavior like in Example \ref{ex:schwarz:2}.
\end{example}

\bibliographystyle{amsplain}
\bibliography{Bibly}

@article{EC,
  author = {E. Calabi},
  title = {\textit{Improper affine hyperspheres of convex type and a generalization of a theorem by K. Jörgens}},
  journal = {\textnormal{Michigan Math. J.}},
  volume = {5},
  pages = {105--126},
  year = {1958}
}

@article{Jor,
  author = {K. Jörgens},
  title = {\textit{Über die Lösungen der Differentialgleichung $rt - s^2 = 1$}},
  journal = {\textnormal{Math. Ann.}},
  volume = {127},
  pages = {130--140},
  year = {1954}
}

@article{Nit,AUTHOR = {Nitsche, J.C.C.},
     TITLE = {\textit{Elementary proof of {B}ernstein's theorem on minimal surfaces}},
   JOURNAL = {\textnormal{Ann. of Math.}},
    VOLUME = {66},
      YEAR = {1957},
     PAGES = {543--544},
}

@article{Hit1,
author={N.J. Hitchin},
title={\textit{The moduli space of special Lagrangian submanifolds}},
journal={\textnormal{Ann. Sc. Norm. Sup. Pisa}},
volume={25},
pages={503--515},
year={1997}}

@article{Hit2,
author={N.J. Hitchin},
title={\textit{The moduli space of complex Lagrangian submanifolds}},
journal={Asian J. Math.},
volume={3},
pages={77--92},
year={1999}
}

@phdthesis{JSAR,
  author = {J.R. {Santiago Arellano}},
  title = {\textit{Geometry of Monge-Ampère metrics}},
  school = {\textnormal{University of Münster}},
  year = {2026},
  note = {available at \url{https://doi.org/10.17879/98838797940}}
}

@book{Besse,
  author = {Besse, A.L.},
  title = {\textit{Einstein manifolds}},
  publisher = {\textnormal{Springer}},
  address = {\textnormal{Berlin}},
  year = {2007}
}

@article{Bie,
    AUTHOR = {Bielawski, R.},
     TITLE = {\textit{Complete hyper-{K}\"ahler {$4n$}-manifolds with a local
              tri-{H}amiltonian {$\mathbf{R}^n$}-action}},
   JOURNAL = {\textnormal{Math. Ann.}},
    VOLUME = {314},
      YEAR = {1999},
     PAGES = {505--528},
}

@article{PP,
  author = {H. Pedersen and Y. Poon},
  title = {\textit{Hyper-{K}ähler metrics and a generalization of the Bogomolny equations}},
  journal = {\textnormal{Comm. Math. Phys.}},
  volume = {117},
  pages = {569--580},
  year = {1988}
}

@article{HKLR,
  author = {N.J. Hitchin and A. Karlhede and U. Lindstr{\"o}m and M. Roček},
  title = {\textit{Hyper-{K}{\"a}hler metrics and supersymmetry}},
  journal = {\textnormal{Comm. Math. Phys.}},
  volume = {108},
  pages = {535--589},
  month = {Nov},
  year = {1987},
  doi = {10.1007/BF01214418}
}

@article{VC,
  author  = {Cortés, V.},
  title   = {\textit{A holomorphic representation formula for parabolic hyperspheres}},
  journal = {\textnormal{Banach Center Publ.}},
  volume  = {57},
  pages   = {11--16},
  year    = {2002},
  url     = {http://eudml.org/doc/282020}
}

@article{Lewy1,
  author  = {H. Lewy},
  title   = {\textit{A priori limitations for solutions of Monge-Ampère equations I}},
  journal = {\textnormal{Trans. Amer. Math. Soc.}},
  volume  = {37},
  pages   = {417--434},
  year    = {1935}
}

@article{Lewy2,
  author  = {H. Lewy},
  title   = {\textit{A priori limitations for solutions of Monge-Ampère equations II}},
  journal = {\textnormal{Trans. Amer. Math. Soc.}},
  volume  = {41},
  pages   = {365--374},
  year    = {1937},
  doi     = {10.1090/S0002-9904-1936-06397-4}
  }

@article{Freed,
  author  = {Freed, D.S.},
  title   = {\textit{Special Kähler manifolds}},
  journal = {\textnormal{Comm. Math. Phys.}},
  volume  = {203},
  pages   = {31--52},
  year    = {1999},
  doi     = {10.1007/s002200050604},
  url     = {https://doi.org/10.1007/s002200050604}
}

@article{Lu,
  author  = {Z. Lu},
  title   = {\textit{A note on special Kähler manifolds}},
  journal = {\textnormal{Math. Ann.}},
  volume  = {313},
  pages   = {711--713},
  year    = {1999},
  doi     = {10.1007/s002080050278}
}

@article{ACD,
  author  = {Alekseevsky, D.V. and Cortés, V. and Devchand, C.},
  title   = {\textit{Special complex manifolds}},
  journal = {\textnormal{J. Geom. Phys.}},
  volume  = {42},
  pages   = {85--105},
  year    = {2002}
}

@article{BC,
  author  = {Baues, O. and Cortés, V.},
  title   = {\textit{Realisation of special Kähler manifolds as parabolic spheres}},
  journal = {\textnormal{Proc. Amer. Math. Soc.}},
  volume  = {129},
  pages   = {2403--2407},
  year    = {2001},
  doi     = {10.1090/S0002-9939-00-05981-5}
}

@article{Pogo,
  author  = {Pogorelov, A.V.},
  title   = {\textit{On the improper convex affine hyperspheres}},
  journal = {\textnormal{Geom. Ded.}},
  volume  = {1},
  pages   = {33--46},
  year    = {1972},
  doi     = {10.1007/BF00147379}
}

@misc{Bryantexample,
  author = {Bryant, R.L.},
  title = {\textit{Answer to ``Riemannian surfaces with an explicit distance function?''}},
  year = {2011},
  howpublished = {MathOverflow},
  note = {available at \url{https://mathoverflow.net/a/61846}}
}
\end{document}